\documentclass{article}
\usepackage[utf8]{inputenc}
\usepackage[T1]{fontenc}
\usepackage{lmodern}
\usepackage{graphicx}
\usepackage{multicol}

\usepackage{pgf,tikz,pgfplots}
\pgfplotsset{compat=1.14}
\usepackage{mathrsfs}
\usetikzlibrary{arrows}
\usepackage{amsmath,amssymb,amsthm}
\theoremstyle{theorem}
\newtheorem{theorem}{Theorem}[section]
\theoremstyle{definition}
\newtheorem{definition}[theorem]{Definition}
\theoremstyle{definition}
\newtheorem{example}[theorem]{Example}
\theoremstyle{remark}
\newtheorem{remark}[theorem]{Remark}
\theoremstyle{theorem}
\newtheorem{proposition}[theorem]{Proposition}
\theoremstyle{theorem}
\newtheorem{lemma}[theorem]{Lemma}
\theoremstyle{theorem}
\newtheorem{corollary}[theorem]{Corollary}

\newcounter{exocmpt}[section]
\newcounter{partcmpt}[exocmpt]
\newcounter{questioncmpt}[exocmpt]
\newcounter{subquestioncmpt}[questioncmpt]

\theoremstyle{definition}

\newcommand{\cp}{\mathbb{C}\mathbb{P}}

\newcommand{\cmplx}{\mathbb{C}}
\newcommand{\real}{\mathbb{R}}

\newcommand{\ante}{^{-1}}
\newcommand{\transp}{{}^{t}}
\newcommand{\orth}{^{\bot}}
\newcommand{\etoile}{^{\ast}}

\newcommand{\class}{\mathcal{C}}

\newcommand{\ellipse}{\mathcal{E}}

\newcommand{\inftyline}{\overline{\cmplx}_{\infty}}

\newcommand{\cat}{\text{Cat}}

\begin{document}

\title{Complex caustics of the elliptic billiard}
\author{}
\date\today
\maketitle

\noindent Corentin FIEROBE, École Normale Supérieure de Lyon, Unité de Mathématiques Pures et Appliquées, UMR CNRS 5669, 46, allée d’Italie, 69364 Lyon Cedex 07, France ; \textit{e-mail : corentin.fierobe@ens-lyon.fr}

\tableofcontents
\newpage

\section{Introduction}

It is a well known fact that on a real ellipse, a billiard trajectory remains tangent to the same confocal conic (ellipse or hyperbola) after successive bounces. This is true in particular for periodic orbits. See for exemple \cite{tabachnikovgeometry}, Chapter $4$.

In this present paper, we study the caustics of a complex reflection law, which will be defined later (in Section \ref{sec_complex_reflection_law}). For now, one should just know that it is a natural generalization of the real law and that it was used to prove important results. More precisely, let $q$ be the non-degenerate complex quadratic form $q(x,y)=x^2+y^2$ on $\cmplx^2$. Then to each complex line $L$ passing through a point $A$, one can associate its orthogonal line (for $q$) $L'$ passing through $A$. Then in the case when $L\neq L'$ (non-isotropic case) we can define the reflection of lines passing through $A$ as the usual way. In the case when $L=L'$, the reflection of lines is defined as a limit of lines reflected in non-isotropic cases. We give more details in Section \ref{sec_complex_reflection_law}.

It was introduced by Glutsyuk in order to study Ivrii's conjecture on periodic orbits together with its analogues for pseudo-billiards and complex billiards (\textit{cf} \cite{glut1, glut2, glut3}). Another use of it was made by Romaskevich in \cite{romaskevich2014}, to prove that the set of incenters of triangular orbits in an ellipse is also an ellipse. 

The idea to approach a real problem by studying its extension to complex domain is not new. For example a proof of Poncelet theorem was given in \cite{poncGH} by generalizing it to complex caustics in $\cp^2$. Note that Poncelet's theorem was also studied on other fields, see for example \cite{JC17} and \cite{HK2014}. Since it is a key point in this article, let us recall it.

\begin{theorem}[Poncelet, \cite{poncGH}]
Let $C$ and $D$ be two conics of $\cp^2$ intersecting transversally. If there is a $n$-sided polygon inscribed in $C$ and circumscribed about $D$, then for each point $p$ of $C$ there is such a polygon having $p$ as a vertex.
\end{theorem}

Billards with other reflection laws were already studied, see \cite{khesin} for an introduction to this topic, and \cite{ADR2019} for a precise study in the Minkovski settings. Furthermore, new discoveries about the elliptic billard are made even recently, see for example Reznik's github page \cite{reznik_github}, and \cite{AST,BT} for mathematical proofs of some of its experimental results.

We denote by $\class_{\lambda}$ the (complex) conic given in the affine chart $(x,y)$ by equation 
$$\mathcal{C}_{\lambda}:\,\frac{x^2}{a^2+\lambda}+\frac{y^2}{b^2+\lambda}=1$$
where $x,y\in\cmplx^2$ and $\lambda\in\cmplx\setminus\{-a^2, -b^2\}$ ; let us also define $\ellipse = \class_0$ that is
$$\ellipse:\,\frac{x^2}{a^2}+\frac{y^2}{b^2}=1.$$
We first prove the following theorem (see Fig. \ref{fig_main_theorem_caustics}):

\begin{theorem}[Complex caustics]
\label{main_theorem_caustics}
Let $T=(M_0,\ldots,M_n)$ be a $(n+1)-$uplet of points of $\ellipse$ such that any two consecutive points are dictinct. The following statements are equivalent:
\begin{itemize}
	\item there is a $\lambda$, such that $a^2+\lambda,b^2+\lambda\neq0$ and for each $j\in\{1,\ldots,n-1\}$ the lines $M_{j-1}M_j$ and $M_jM_{j+1}$ realize the two tangent lines to $\class_{\lambda}$ going through $M_j$;
	\item $T$ is a non-degenerate non-isotropic piece of complex billiard trajectory whose sides do not contain the real nor the complex foci.
\end{itemize}
If this $\lambda$ exists, then it is unique.
\end{theorem}

\begin{figure}[!h]
\centering
\begin{tikzpicture}[line cap=round,line join=round,>=triangle 45,x=1cm,y=1cm]
\clip(-5,-2.1) rectangle (5,2.1);
\draw [rotate around={0:(0,0)},line width=1pt] (0,0) ellipse (2.8284271247461907cm and 2cm);
\draw [rotate around={0:(0,0)},line width=1pt] (0,0) ellipse (2.57838050536101cm and 1.627281791954208cm);
\draw [line width=1pt] (-2.822840787497323,-0.12563814795175854)-- (-1.869312067799702,1.500945101124356);
\draw [line width=1pt] (-1.869312067799702,1.500945101124356)-- (1.393051102240205,1.740604582687773);
\draw [line width=1pt] (1.393051102240205,1.740604582687773)-- (2.818498874725576,0.1674277354124261);
\draw [line width=1pt] (2.818498874725576,0.1674277354124261)-- (1.928469151211218,-1.463045920814977);
\begin{scriptsize}
\draw [fill=black] (-2,0) circle (2.5pt);
\draw[color=black] (-1.9906447176999233,0.4527745617943184) node {$F_1$};
\draw [fill=black] (2,0) circle (2.5pt);
\draw[color=black] (2.030966426407765,0.4773715718500229) node {$F_2$};
\draw[color=black] (-2.1013312629505934,1.7) node {$M_1$};
\draw[color=black] (1.4898322051822654,1.9) node {$M_2$};
\draw[color=black] (3.2,0.2) node {$M_3$};
\draw[color=black] (-3.2,-0.1) node {$M_0$};
\draw[color=black] (2.030966426407765,-1.7) node {$M_4$};
\draw[color=black] (-1,-1.3) node {$\class_{\lambda}$};
\draw[color=black] (0,-1.85) node {$\ellipse$};
\end{scriptsize}
\end{tikzpicture}
\caption{The confocal caustic $\class_{\lambda}$ inscribed in a piece of billiard trajectory.}
\label{fig_main_theorem_caustics}
\end{figure}
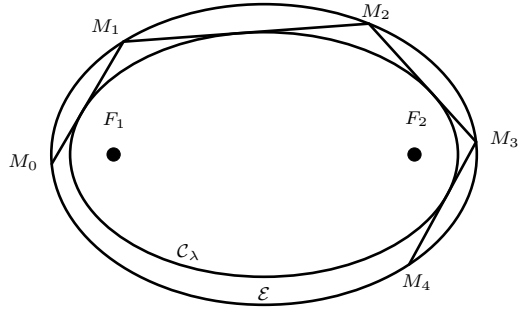

The definition of non-degenerate and non-isotropic complex billiard orbits will be defined later, in Subsection \ref{subsec_def_billiard}. For now, the reader can think about them as "good" orbits. Thus Theorem \ref{main_theorem_caustics} generalizes the theorem of the existence of caustics in the elliptic billiard to its complexification.

For periodic orbits, we improve Theorem \ref{main_theorem_caustics} as follows:

\begin{theorem}
\label{main_theorem_orbits}
Fix $n\geq 3$. There exist $N$ confocal conics $\gamma^n_1,\ldots,\gamma^n_N$ satisfying the following properties:
\begin{enumerate}
	\item any $n-$sided polygon inscribed in $\ellipse$, circumscribed about a $\gamma^n_j$ and having its consecutive vertices distinct, is a periodic billiard orbit (which is non-degenerate non-isotropic and non-flat);
	\item any non-degenerate non-isotropic periodic billiard orbit which doesn't pass through a (complex or real) focus of $\ellipse$ is circumscribed about one of the $\gamma^n_j$;
	\item for each point $p$ of $\ellipse$ and each $j$, one can find an $n-$sided polygon inscribed in $\ellipse$ and circumscribed about a $\gamma^n_j$, having $p$ as a vertex.
\end{enumerate}
Furthermore, $N\leq n^2/4$ if $n$ is even, and $N\leq (n^2-1)/4$ if $n$ is odd. For $n=3$ one has $N=2$. For $n=4$, $N=3$ in the case when $a\neq\sqrt{2}b$ and $N=2$ otherwise. 
\end{theorem}

\begin{remark}
Theorem \ref{main_theorem_orbits} gives a meaning to another result which can be found in \cite{DragRad2019}. In this paper, at pages 17-18, the authors use a similar method to ours coming from elliptic functions theory to give a classification of the real caustics of the elliptic billard. Similarly to us, the caustics are related to the roots of a specific polynomial, and in their case some on these roots do not correspond to any real caustics. For example in the case of triangular orbits, one can find two roots: one of them corresponds to the usual smaller confocal ellipse, the other one defines a confocal ellipse which is bigger than the billard table and cannot be a caustic for real triangular orbits by convexity. Our result allows to understand this other root also as another caustic, but for the reflection law extended to the complex domain.
\end{remark}

A rather interesting application of Theorem \ref{main_theorem_orbits} is the classification of 3-periodic degenerate orbits (defined as the limits of non-degenerate orbits) - which is related to a result in \cite{romaskevich2014}:

\begin{corollary}
There are exacty $8$ degenerate triangular orbits, given by an isotropic tangency point $\alpha$ of $\ellipse$ and the point $\beta$ of $\ellipse$ such that the line $\alpha\beta$ is tangent to a conic $\gamma^3_j$ and non-isotropic.
\end{corollary}

\subsection{Structure of the document and main arguments}

We exclude the case when $\ellipse$ is a circle. The document has the following structure:

\begin{enumerate}
	\item \textbf{Section \ref{sec_complex_reflection_law}} gives more details about the complex reflection law and what complex billiard trajectories are.
	\item We present a natural and short proof of Theorem \ref{main_theorem_caustics} in \textbf{Section \ref{sec:algebraic_proof}}: it uses Zariski topology and classical results on real billards.
	\item In \textbf{Section \ref{sec_invariant}}, we present another proof of Theorem \ref{main_theorem_caustics} that simultaneously
yields new results on Joachimstal invariant. The latter is a quantity $P(M,v)$ depending on a vertex $M\in\ellipse$ of a polygon $T$ inscribed in $\ellipse$ and a vector $v$ directing a side of this polygon starting at $M$. We show that this quantity: a) doesn't depend on the vertex $M$ chosen or $v$ if and only if $T$ is a billiard trajectory ; b) is directly related with the existence of a conic inscribed in $T$.
	\item In \textbf{Section \ref{sec_n_orbits}} we prove Theorem \ref{main_theorem_orbits}. We use a theorem of Cayley to establish the existence of confocal conics inscribed in an $n-$sided polygon which is itself inscribed in $\ellipse$, and also to bound their number. This theorem reduces the proof to the computation of a certain polynomial's degree.
	\item The cases of $3-$ and $4-$periodic orbits are studied in \textbf{Sections \ref{sec_3_orbits} and \ref{sec_4_orbits}}, where we compute the exact value of the $N$ appearing in Theorem \ref{main_theorem_orbits}, getting similar results as in \cite{DragRad2019}.
\end{enumerate}

\subsection{Notations and usual properties of conics}
\label{subsec_notations_conics}

In the whole article, we are dealing with the complexification of the real ellipse $\ellipse$ and with the complex ellipses $\class_{\lambda}$. Thus we recall some results about these objects.

The ellipse $\ellipse$ has four \textit{isotropic tangent lines} (which are tangent lines directed by the vectors $(1,\pm i)$, cf \cite{berger90}, Volume 2, Section 17.4.3). The name \textit{isotropic} is due to the fact that $(1,\pm i)$ are isotropic vectors of the complex quadratic form $q(x,y)=x^2+y^2$, i.e., q vanishes on them.

A simple computation shows that their corresponding tangency points have coordinates
\begin{equation}\label{equ_tang_points}
\left(\pm\frac{a^2}{\sqrt{a^2-b^2}},\pm \frac{ib^2}{\sqrt{a^2-b^2}}\right).
\end{equation}
where the signs $\pm$ are independent. This allows us to extend the definition of a \textit{focus of an ellipse} as an intersection point of its non-parallel isotropic tangent lines, cf \cite{berger90}, Volume 2, Section 17.4.3. In our case, the ellipse $\ellipse$ has four foci:
\begin{itemize}
	\item two \textit{real} foci of coordinates $(\pm c,0)$, where $c=\sqrt{a^2-b^2}$ ;
	\item two \textit{complex} foci of coordinates $(0,\pm ic)$.
\end{itemize}

The \textit{foci lines} of $\ellipse$ are defined as two distinct lines : the line joining the complex foci and the line joining the real foci.

\begin{remark} We see here that $\ellipse$ and $\class_{\lambda}$ have the same foci. Hence they have the same isotropic tangent lines !
\end{remark}

The following result will be needed, which generalizes a well known result in the real case (concerning the intersection of an ellipse with a confocal hyperbola). To state it, we recall that $q(x,y) = x^2+y^2$ is a complex quadratic form, whose associated bilinear form is $b$ defined by 
$$b\left(u,v\right) = u_xv_x+u_yv_y$$
for all $u=(u_x,u_y),v=(v_x,v_y)\in\cmplx^2$. Two vector spaces $F$, $G$ of $\cmplx^2$ are said to be orthogonal if 
$$b(u,v)=0$$
for all $u\in F$ and $v\in G$, and this definition extends naturally to lines in $\cmplx^2$: two lines are orthogonal if their directions are orthogonal.

\begin{lemma}
\label{lemma_common_points}
Let $\lambda \neq 0$. Then $\ellipse$ and $\class_{\lambda}$ have four common points, whose coordinates $(x,y)$ are such that
$$x^2 = \frac{a^2(a^2+\lambda)}{a^2-b^2}\qquad\text{and}\qquad y^2 = \frac{b^2(b^2+\lambda)}{b^2-a^2}.$$
The tangent lines to $\class_{\lambda}$ and $\ellipse$ at these points are \textit{orthogonal} (for $q$).
\end{lemma}

\begin{proof}
The coordinates of the common points are obtained by solving the system composed by the equations of $\class_{\lambda}$ and $\ellipse$.

Then, since the equations of the tangent lines of $\ellipse$ and $\class_{\lambda}$ in one of these common points can be computed, it is not difficult to check that both tangent lines are orthogonal.
\end{proof}

Finally, we will name by $S$ the point of coordinates $(-a,0)$ of $\ellipse$.

\section{Complex reflection law}
\label{sec_complex_reflection_law}

Here we introduce the notion of complex billiards, which is somewhat similar to pseudo-euclidean billiards studied by V. Dragovic and M. Radnovic in \cite{dragovic}.

Considering an affine chart whose coordinates will be denoted by $(x,y)$, we have the inclusion $\real^2 \subset \cmplx^2 \subset \cp^2$, and $\cp^2 = \cmplx^2 \sqcup \inftyline$, where $\inftyline$ is the infinity line. As introduced and explained in \cite{glut1}, and studied in \cite{romaskevich2014}, the reflection law on an algebraic (analytic) curve in $\real^2$ can be extended to $\cp^2$ by considering the complexified version of the canonical euclidean quadratic form, that is the complex-bilinear non-degenerate quadratic form $q$. It leads to construct a notion of symmetry with respect to lines of $\cmplx^2$. Just note that $q$ has two isotropic subspaces of dimension $1$ (namely $\cmplx (1,i)$ and $\cmplx (1,-i)$). 

\subsection{Definition}

\begin{definition}
Define the \textit{cyclic} points of $\cp^2$ as the points $I = (1:i:0)$ or $J = (1:-i:0)$.
\end{definition}

\begin{definition}[\cite{glut1}, definition 1.2] A line in $\cp^2$ is said to be \textit{isotropic} if it contains either $I$ or $J$ and \textit{non-isotropic} if not. (Thus, the infinity line is automatically isotropic.)
\end{definition}

\begin{definition}[\cite{glut1}, definition 2.1]
The \textit{symmetry} with respect to a line $L$ is defined by the two following points:
	\begin{itemize}
		\item the \textit{symmetry acting on $\cmplx^2$ (on points and on lines)}: it is the unique non-trivial complex-isometric involution (isometric for the form $q$) fixing the points of the line $L$, if $L$ is non-isotropic ;
		\item the \textit{symmetry acting on lines}: if $L$ is an isotropic line passing through a point $x\notin\inftyline$, two lines $\ell$ and $\ell'$ going through $x$ are called symmetric with respect to $L$ if there exist sequences of lines $(L_n)_n$, $(\ell_n)_n$, $(\ell'_n)_n$ through points $x_n$ such that $L_n$ is non-isotropic, $\ell_n$ and $\ell'_n$ are symmetric with respect to $L_n$, $\ell_n \to \ell$, $\ell'_n \to \ell'$, $L_n \to L$ and $x_n \to x$, when $n\to\infty$.
	\end{itemize}
\end{definition}

We recall now Lemma $2.3$ \cite{glut1} which gives an idea of this notion of symmetry in the case of an isotropic line through a finite point.

\begin{lemma}[\cite{glut1}, Lemma 2.3]
\label{lemma_glutsyuk}
If $L$ is an isotropic line through a finite point $x$ and $\ell$, $\ell'$ are two lines passing through $x$, then $\ell$ and $\ell'$ are symmetric with respect to $L$ if and only if either $\ell=L$, or $\ell'=L$.
\end{lemma}

\subsection{General billiard orbits}

\label{subsec_def_billiard}

\begin{definition}
A \textit{piece of non-degenerate complex billiard trajectory} is an ordered set of points $(M_0,M_1,\ldots,M_n)$ on $\ellipse$ such that
\begin{enumerate}
	\item $M_j\neq M_{j+1}$ for all $j\in\{0,\ldots,n-1\}$;
	\item for all $j\in\{1,\ldots,n-1\}$ the line $T_{M_j}\ellipse$ is non-isotropic;
	\item for all $j\in\{1,\ldots,n-1\}$, the lines $M_{j-1}M_j$ and $M_jM_{j+1}$ are symmetric with respect to the tangent line $T_{M_j}\ellipse$.
\end{enumerate}
We say that a piece of non-degenerate trajectory is a non-degenerate \textit{periodic orbit} or just an \textit{orbit} when $M_n=M_0$ and when the above statements are also true for $M_0$.

The $M_j$'s are called \textit{vertices} of the trajectory, and the lines $M_jM_{j+1}$ \textit{sides} of the trajectory.
\end{definition}

Computations will be easier with the following definitions.

\begin{definition}
A piece of non-degenerate trajectory is said to be :
\begin{enumerate}
	\item \textit{finite} if none of its vertices belong to the infinity line, and \textit{infinite} if one of them belong to $\inftyline$;
	\item \textit{isotropic} if all of its sides are isotropic, and \textit{non-isotropic} if none of its sides is isotropic.
	\item \textit{flat} if all of its sides coincide with one of both foci lines of $\ellipse$.
\end{enumerate}
\end{definition}

\begin{remark}
It is not difficult to see that a piece of trajectory has a side on a foci line if and only if it has all its sides on this foci line.
\end{remark}

Then we will use the following proposition, which follows from the fact that the reflection law with respect to a non-isotropic line permutes two isotropic lines, \cite{glut2}, corollary $2$ (there exist exactly two distinct isotropic lines passing through any point $x\notin\inftyline$).

\begin{proposition}\label{prop_non_isotropic_orbits}
A non-degenerate periodic orbit with an odd number of sides is non-isotropic and none of its sides lie on a foci line.
\end{proposition}

When $n=3$, periodical orbits are called triangular. The proposition \ref{prop_non_isotropic_orbits} implies that there are no non-degenerate isotropic triangular orbits.

\begin{definition}[\cite{glut1}]
A \textit{degenerate triangular complex orbit} on a complex conic $\ellipse$ is an ordered triple of points in $\ellipse$ which is the limit of non-degenerate periodic triangular orbits and which is not a non-degenerate triangular orbit. We define the \textit{sides} of a degenerate orbit as the limits of the sides of the non-degenerate orbits which converge to it. Note that the sides of a degenerate orbit are uniquely defined, see the following more precise proposition.
\end{definition}

\begin{proposition}[\cite{romaskevich2014}, lemma 3.4]
\label{lemma_olga}
A degenerate triangular orbit of an ellipse $\ellipse$ has an isotropic side $A$ which is tangent to $\ellipse$, and two coinciding non-isotropic sides $B$.
\end{proposition}

\section{An algebraic proof of Theorem \ref{main_theorem_caustics}}
\label{sec:algebraic_proof}

In this section we give a short algebraic proof of Theorem \ref{main_theorem_caustics}, noticing that the theorem is true for real orbits and using Zariski topology to conclude.

\begin{proposition}
\label{prop:dual_line_conics}
For every projective line $L$ passing through neither real, nor complex foci there exists a unique $\lambda\notin\{-a^2,-b^2\}$ such that the conic $\mathcal{C}_{\lambda}$ is tangent to $L$.
\end{proposition}

\begin{proof} We prove this result by taking projective duality given by polar duality with respect to the Euclidean metric. 

Denote by $F_1$ and $F_2$ the real foci of all $\mathcal{C}_{\lambda}$, and by $G_1$ and $G_2$ their complex foci: $F_1$, $F_2$ are on the $x$-axis, $G_1$, $G_2$ on the $y$-axis (see Subsection \ref{subsec_notations_conics}, or \cite{berger90}, Volume 2, Section 17.4). Hence the dual line $F_i\etoile$ of $F_i$ is vertical and the dual line $G_i\etoile$ of $G_i$ is horizontal.

Then note that any line $F_iG_j$ is tangent to any $\mathcal{C}_{\lambda}$ (again, see Subsection \ref{subsec_notations_conics}, or \cite{berger90}, Volume 2, Section 17.4.3). It induces a point $x_{ij}$ defined as the dual of the line $F_iG_j$, which belongs to all dual conics $\mathcal{C}_{\lambda}\etoile$ of $\mathcal{C}_{\lambda}$. Hence, the family $\mathcal{C}_{\lambda}\etoile$ is a pencil of conics through $4$ distinct points $x_{ij}$. The conics $\mathcal{C}_{\lambda}\etoile$ of the dual pencil are given by the equation $(a^2 + \lambda)\tilde{x}^2 + (b^2+\lambda)\tilde{y}^2 = 1$. They are smooth for every $\lambda\notin\{-a^2,-b^2\}$.

Now, to any line $L$ passing through neither real, nor complex foci corresponds a dual point $y_L$ different from the $x_{ij}$. By polar duality, Proposition \ref{prop:dual_line_conics} is equivalent to say that there is a unique conic passing through $y_L$ in the pencil of conics defined by the dual conics of the $\mathcal{C_{\lambda}}$. And this is a classical result on pencils of conics which is not difficult to prove, see for example \cite{Flatto}.

To see that $\lambda\notin\{-a^2,-b^2\}$, suppose the contrary, for example $\lambda=-b^2$. The conic $\mathcal{C}_{\lambda}\etoile$ is a pair of vertical lines passing through the $x_{ij}$, hence both lines are the dual lines $F_1\etoile$ and $F_2\etoile$ of the real foci. But $L$ do not contain any foci, thus its dual point $L\etoile$ do not belong to $F_1\etoile\cup F_2\etoile=\mathcal{C}_{\lambda}\etoile$. 

The same statement holds for complex foci and $-b^2$ replaced by $-a^2$.
\end{proof}

\begin{remark}
By proof of Proposition \ref{prop:dual_line_conics}, we see in fact that if $L$ goes through a real focus, its dual point $L\etoile$ belongs to a $\mathcal{C}_{\lambda}\etoile$ only when $\lambda=-b^2$. Hence $L$ is never tangent to a $\mathcal{C}_{\lambda}$ for $\lambda\neq-b^2$.
\end{remark}

\begin{corollary}
\label{cor:zariski_billard}
The collections of $(A, L, \lambda)$ such that $A$ lies in the given ellipse $\ellipse$ and $L$ is a line through $A$ that is tangent to $\mathcal{C}_{\lambda}$ form an irreducible two-dimensional algebraic surface, in which the real part (real lines tangent to real confocal
ellipses) is Zariski dense. The image of $L$ under reflection from the tangent line $T_A\ellipse$ is again tangent to $\mathcal{C}_{\lambda}$.
\end{corollary}

\begin{proof}
The first statement of Corollary \ref{cor:zariski_billard} follows immediately from Proposition \ref{prop:dual_line_conics}. The second statement follows from the fact that it obviously holds on the Zariski dense real part and from algebraicity of the tangency condition for the reflected line.
\end{proof}

\begin{proof}[Proof of Theorem \ref{main_theorem_caustics}]
If the first statement of Theorem \ref{main_theorem_caustics} is realized, then for each $j$, by Corollary \ref{cor:zariski_billard} the line $M_{j-1}M_j$ is reflected into the $M_{j+1}M_j$ under reflection from the tangent line $T_{M_j}\ellipse$. Hence $T$ is a billard orbit.

If the second statement of Theorem \ref{main_theorem_caustics} is true, by Proposition \ref{prop:dual_line_conics} there is a unique $\lambda\notin\{-a^2,-b^2\}$ such that the conic $\mathcal{C}_{\lambda}$ is tangent to the line $M_0M_1$. By Corollary \ref{cor:zariski_billard}, each $M_{j}M_{j+1}$ is again tangent to $\mathcal{C}_{\lambda}$. 

Note that if you fix $j$, $M_{j-1}M_j$ and $M_{j+1}M_j$ are the two tangent lines to $\mathcal{C}_{\lambda}$ passing through $M_j$: indeed, suppose $M_{j-1}M_j=M_{j+1}M_j$ and take any tangent line $T$ to $\mathcal{C}_{\lambda}$ passing through $M_j$. By Corollary \ref{cor:zariski_billard}, it is reflected in a tangent line $T'$ to $\mathcal{C}_{\lambda}$. In the case when $T=T'$ the line $T$ is orthogonal to $T_{M_j}\ellipse$, and so is the line $M_jM_{j+1}$ for the same reasons. If $T\neq T'$, we have $T$ or $T'=M_jM_{j+1}$ or we would have three distinct tangent lines to $\ellipse$ passing through $M_j$ which is impossible. Hence $T=T'=M_jM_{j-1}=M_jM_{j+1}$.
\end{proof}

\section{Theorem \ref{main_theorem_caustics} and Joachimsthal invariant}
\label{sec_invariant}


\begin{remark}
This section was inspired by the study of billiards in conics conducted in Chapter $4$ - \textit{Billards inside Conics and Quadrics} of \cite{tabachnikovgeometry}. In this book, Theorem $4.4$ shows that for a set of points and directions defined as successive billiard reflections on a real conic $\ellipse$, there is an \textit{invariant} quantity. Known as Joachimsthal invariant, it is defined by
$$\frac{xv_x}{a^2}+\frac{yv_y}{b^2}$$
where $(x,y)$ are the coordinates of a vertex of an orbit, and $v$ a unitary vector having this vertex as starting point and pointing toward the next vertex. The author, Tabachnikov, further explains, without proving it, that one can find such an invariant if and only if one can find a conic tangent to the orbit.

In our case, Joachimsthal invariant doesn't work anymore and we need to change it a little bit: a square power appears, and we have to handle the case of isotropic directions, for which unitary vectors cannot be found (that are vectors $v$ such that $q(v)=1$).
\end{remark}

\subsection{A billiard invariant}
\label{subsec_invariant}

\begin{figure}[!h]
\centering

\begin{tikzpicture}[line cap=round,line join=round,>=triangle 45,x=1cm,y=1cm]
\draw [rotate around={0:(0,0)},line width=1pt] (0,0) ellipse (3.6055512754639882cm and 2cm);
\draw [->,line width=0.2pt] (-3.1850116065471195,-0.9373778593108506) -- (-2.6280785187343483,-0.10682049482641465);
\draw [->,line width=0.2pt] (-1.3064833284084274,1.8640816516113508) -- (-1.8634164162211988,1.033524287126915);
\draw [->,line width=0.2pt] (-1.3064833284084274,1.8640816516113508) -- (-0.4567448009853311,1.3368772936885167);
\draw [->,line width=0.2pt] (3.1938575709826025,-0.9280705901993418) -- (2.3441190435595054,-0.40086623227650797);
\draw [line width=1pt] (-3.1850116065471195,-0.9373778593108506)-- (-1.3064833284084274,1.8640816516113508);
\draw [line width=1pt] (-1.3064833284084274,1.8640816516113508)-- (3.1938575709826025,-0.9280705901993418);
\begin{scriptsize}
\draw[color=black] (-3.4,1) node {$\ellipse$};
\draw[color=black] (-3.338207671594807,-1.1) node {$M_0$};
\draw[color=black] (-1.3979696107556636,2.05) node {$M_1$};
\draw[color=black] (3.3336486791209254,-1.1) node {$M_2$};
\draw[color=black] (-3.063645681853419,-0.33217314749790483) node {$v_0$};
\draw[color=black] (-1.5,1.2) node {$v_1'$};
\draw[color=black] (-0.8,1.3) node {$v_1$};
\draw[color=black] (2.9126536281841306,-0.44199794339445997) node {$v_2'$};
\end{scriptsize}
\end{tikzpicture}
\caption{In Proposition \ref{prop_invariant}, we consider all quantities $P(M_0,v_0)$, $P(M_1,v_1')$, $P(M_1,v_1)$ and $P(M_2,v_2')$.}
\label{fig_prop_invariant}
\end{figure}
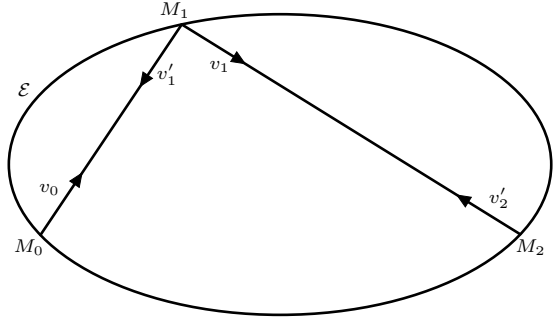

\begin{proposition}
\label{prop_invariant}
Let $T= (M_0,M_1,M_2)$ be a piece of non-degenerate and non-isotropic trajectory on $\ellipse$ with $M_0$ finite. Then the quantity
$$P(M_j,v) = \frac{\left(\frac{x_jv_x}{a^2}+\frac{y_jv_y}{b^2}\right)^2}{q(v)}$$
where $(x_j,y_j)$ are the coordinates of a finite vertex $M_j$ and $v=(v_x,v_y)$ is a directing vector of $M_{j-1}M_j$ or $M_jM_{j+1}$, is independent on the index $j\in\{0,1,2\}$ of the \textbf{finite} vertex chosen and on $v$ (see Fig. \ref{fig_prop_invariant}). 
\end{proposition}

\begin{remark}
For periodic orbits with an odd number of sides, one can remove the non-isotropic assumption, see Proposition \ref{prop_non_isotropic_orbits}.
\end{remark}

\begin{proof}
Since a non-isotropic piece of trajectory has non-isotropic sides by definition, $q(v)\neq0$ for all $v$ taken like in the proposition we want to prove.

\textbf{First case:} If $M_0$ and $M_1$ are finite, write $M_0 = (x_0,y_0)$, $M_1=(x_1,y_1)$. Take $v_0$ a vector such that $q(v_0)=1$ and directing $M_0M_1$ and $v_1$ vector such that $q(v_1)=1$ and directing $M_1M_2$. Define the matrix
$$A=\left(\begin{matrix}1/a^2&0\\0&1/b^2 \end{matrix}\right).$$ 
Then since $\transp M_j AM_j=1$ and since $A$ is symmetric, we get
$$\transp(M_1-M_0)A(M_1+M_0) = \transp M_1AM_0-\transp M_0AM_1=0.$$
Since $v_0$ is collinear to $M_1-M_0$ we have further $\transp v_0A(M_1+M_0)=0$, thus
\begin{equation}\label{equ_first_invariant}
\transp v_0AM_1=-\transp v_0AM_0.
\end{equation}
But since $M_0M_1$ and $M_1M_2$ are symmetric with respect to the tangent line of $\ellipse$ in $M_1$ which is also orthogonal to $AM_1$ (the gradient in $M_1$ of the bilinear form defining $\ellipse$), we only have two possibilities : either $v_0+v_1$ or $v_0-v_1$ is orthogonal to $AM_1$. Hence
$$\transp(v_0+v_1)AM_1=0$$
or
$$\transp(v_0-v_1)AM_1=0.$$
In both cases
$$\left(\transp v_0AM_1\right)^2=\left(\transp v_1AM_1\right)^2$$
and using equality \eqref{equ_first_invariant}, we get
\begin{equation}\label{equ_second_invariant}
\left(\transp v_0AM_0\right)^2=\left(\transp v_1AM_1\right)^2
\end{equation}
which proves Proposition \ref{prop_invariant} for unitary vectors. For general vectors, it is enough to divide them by a square root of $q(v)$, which explains why there is a $1/q(v)$ in the invariant formula.

\textbf{Second case:} If $M_0$ is finite and $M_1$ infinite (see Fig. \ref{fig_proof_prop_invariant}), then $M_2$ is finite. Indeed, $M_0M_1$ is a finite line and $T_{M_1}\ellipse$ is not isotropic. Hence the line symmetric to $M_0M_1$ with respect to $T_{M_1}\ellipse$ is finite and parallel to $M_0M_1$ and to $T_{M_1}\ellipse$. Thus $M_2$ is finite. And therefore, we need to prove that 
$$P(M_0,v) = P(M_2,v)$$
with $v$ a vector directing $T_{M_1}\ellipse$ (because $v$ directs the lines $M_0M_1$, $M_1M_2$ and $T_{M_1}\ellipse$). But $M_2 =-M_0$ since $T_{M_1}\ellipse$ goes through the origin $O=(0:0:1)$ (by property of a tangent line at an point of $\ellipse$ on $\inftyline$) and the ellipse $\ellipse$ is symmetric across $O$ (see Fig. \ref{fig_proof_prop_invariant}). This implies that $P(M_0,v) = P(M_2,v)$.
\end{proof}

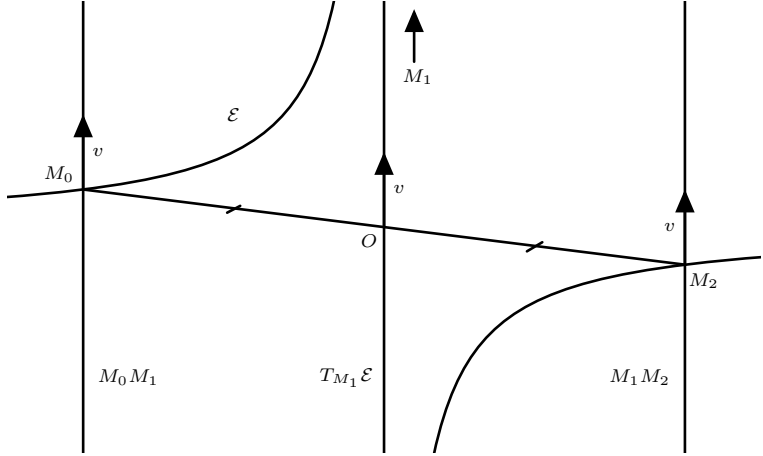
\begin{figure}[!h]
\centering
\begin{tikzpicture}[line cap=round,line join=round,>=triangle 45,x=1cm,y=1cm]
\clip(-5,-3) rectangle (5,3);
\draw [samples=50,domain=-0.99:0.99,rotate around={-45:(0,0)},xshift=0cm,yshift=0cm,line width=1pt] plot ({2*(1+(\x)^2)/(1-(\x)^2)},{2*2*(\x)/(1-(\x)^2)});
\draw [samples=50,domain=-0.99:0.99,rotate around={-45:(0,0)},xshift=0cm,yshift=0cm,line width=1pt] plot ({2*(-1-(\x)^2)/(1-(\x)^2)},{2*(-2)*(\x)/(1-(\x)^2)});
\draw [line width=1pt] (-4,-5.750207336486205) -- (-4,5.688558131117167);
\draw [line width=1pt] (0,-5.750207336486205) -- (0,5.688558131117167);
\draw [line width=1pt] (4,-5.750207336486205) -- (4,5.688558131117167);
\draw [line width=1pt] (-4,0.5)-- (4,-0.5);
\draw [line width=1pt] (-2.085893357114387,0.19163967179694663)-- (-1.9119521473062633,0.282391607349011);
\draw [line width=1pt] (1.8996291458804462,-0.322621296331418)-- (2.1038210008725913,-0.20161871559533223);
\draw [->,line width=1pt] (0.4,2.2) -- (0.4,2.9);
\draw [->,line width=1pt] (0,0) -- (0,1);
\draw [->,line width=1pt] (-4,0.5) -- (-4,1.5);
\draw [->,line width=1pt] (4,-0.5) -- (4,0.5);
\begin{scriptsize}
\draw[color=black] (-4.32034094759505,0.70) node {$M_0$};
\draw[color=black] (0.45,2) node {$M_1$};
\draw[color=black] (4.25,-0.7126491647002291) node {$M_2$};
\draw[color=black] (-0.2,-0.17438473151438424) node {$O$};
\draw[color=black] (-2,1.5) node {$\ellipse$};
\draw[color=black] (-3.8,1) node {$v$};
\draw[color=black] (3.8,0) node {$v$};
\draw[color=black] (0.2,0.5) node {$v$};
\draw[color=black] (-3.4,-2) node {$M_0M_1$};
\draw[color=black] (3.4,-2) node {$M_1M_2$};
\draw[color=black] (-0.5,-2) node {$T_{M_1}\ellipse$};
\end{scriptsize}
\end{tikzpicture}
\caption{A piece of billiard trajectory $(M_0,M_1,M_2)$ on $\ellipse$ with $M_1$ infinite as in the proof of Proposition \ref{prop_invariant}. The points $M_0$ and $M_2$ are symmetric across $O$, hence $M_2=-M_0$ and $P(M_0,v)=P(M_2,v)$. Here $\ellipse$ is represented as an hyperbola which allows us to view the tangent line at the infinity point $M_1$ as the vertical asymptote.}
\label{fig_proof_prop_invariant}
\end{figure}

\begin{corollary}
Let $T= (M_0,\ldots,M_n)$ be a piece of non-degenerate and non-isotropic trajectory on $\ellipse$. Then the quantity $P(M_j,v)$ defined as before doesn't depend on the choice of a finite vertex $M_j$ or on $v$, a directing vector of $M_{j-1}M_j$ or $M_jM_{j+1}$.
Thus we will write $P(M_j,v) = P(T)$.
\end{corollary}

Now we see that the invariant property implies a billiard reflection property.

\begin{lemma}
\label{lemma_carac_billiard}
Let $M$ be a finite point on $\ellipse$ such that the line $T_{M}\ellipse$ is non-isotropic. Let $\ell_1$, $\ell_2$ two non-isotropic lines passing through $M$ and directed by the vectors $v_1$, $v_2$. If
\begin{equation}\label{equ_invariant_points}
P(M,v_1) = P(M,v_2)
\end{equation}
then one of the following cases holds:
\begin{enumerate}
	\item $\ell_1=\ell_2$;
	\item $\ell_1$ and $\ell_2$ are symmetric with respect to $T_{M}\ellipse$.
\end{enumerate}
\end{lemma}

\begin{proof}
Suppose that case $1$ is not true. Let us prove case $2$. 

We can suppose $q(v_1)=q(v_2)=1$. By the equality (\ref{equ_invariant_points}), we have
$$\transp v_1AM = \pm \transp v_2AM$$
hence

$$\transp (v_2\pm v_1)AM = 0.$$
Thus we get that $v_1+v_2$ or $v_2-v_1$ is orthogonal to $AM$ which is orthogonal to the tangent line of $\ellipse$ in $M$. Hence $v_1+v_2$ or $v_1-v_2$ is tangent to $\ellipse$ in $M$. This implies that one of these vectors is fixed by the reflection with respect to $T_M\ellipse$. Therefore this means that the components of the $v_j$'s along the direction of $T_{M}\ellipse\orth$ are the same or have opposite signs. Since the $v_j$'s are unit vectors, their components along the direction of $T_{M}\ellipse$ are also the same or have opposite signs. 

Hence we have only three possibilities: a)$\,v_1$ and $v_2$ are symmetric with respect to $T_{M}\ellipse$, b)$\,v_1$ and $v_2$ are symmetric with respect to $T_{M}\ellipse\orth$, c)$\,v_2=\pm v_1$. Possibility c) cannot happen, otherwise $\ell_1 = \ell_2$. Hence case $2$ is proven.
\end{proof}

\begin{lemma}
\label{lemma_carac_billiard_2}
Let $M_0,M_1,M_2$ be points on $\ellipse$ such that $M_0,M_2$ are finite and $M_1$ infinite. Let $v_j$ be a vector directing the line $M_1M_j$. If
\begin{equation}\label{equ_invariant_points_2}
P(M_0,v_0) = P(M_2,v_2)
\end{equation}
then one of the following cases holds:
\begin{enumerate}
	\item $M_0=M_2$;
	\item $M_0M_1$ and $M_1M_2$ are symmetric with respect to $T_{M_1}\ellipse$.
\end{enumerate}
\end{lemma}

\begin{proof}
Suppose that case $1$ is not true. Let us prove case $2$. 

Since $M_1\in\ellipse$ is infinite, $M_1=(a:\pm ib:0)$. Thus $M_0M_1$ and $M_1M_2$ are directed by $v=(a:\pm ib)$. To simplify, suppose $v=(a,ib)=v_1=v_2$. Thus since $P(M_0,v) = P(M_2,v)$ by equality (\ref{equ_invariant_points_2}), we have
$$\frac{x_0}{a}+i\frac{y_0}{b}=\varepsilon\left(\frac{x_2}{a}+i\frac{y_2}{b}\right)$$
with $\varepsilon\in\{1,-1\}$, $M_j=(x_j,y_j)$. Hence
$$\frac{x_2-\varepsilon x_0}{a}+i\frac{y_2-\varepsilon y_0}{b}=0.$$
We show that $\varepsilon = -1$. Indeed, if $\varepsilon=1$, the latter equality means that $\overrightarrow{M_0M_2}$ is orthogonal to $v'=(1/a,i/b)$, which is orthogonal to $v$. Therefore $\overrightarrow{M_0M_2}$ is colinear to $v$ : but this is impossible, otherwise $M_0, M_1,M_2$ would be three distinct points of $\ellipse$ on the same line.

Thus $\varepsilon=-1$. Then, applying the same arguments, we have $M_0=-M_2$. Hence $M_0M_1$ reflects into $M_1M_2$.
\end{proof}

\subsection{Particular values of the invariant}
\label{subsec_particular_values}

The question we consider here is : \textit{for which non-isotropic $v$ do we have $P(M,v) = b^{-2}$ or $P(M,v) = a^{-2}$ ?}

\begin{proposition}
\label{prop_lien_invariant_foyers}
If $M$ is not a point of isotropic tangency of $\ellipse$, we have :
\begin{itemize}
	\item $P(M,v) = a^{-2}$ if and only if $v$ has the same direction as the line going through $M$ and one of the real foci of $\ellipse$.
	\item $P(M,v) = b^{-2}$ if and only if $v$ has the same direction as the line going through $M$ and one of the complex foci of $\ellipse$.
\end{itemize}
\end{proposition}

\begin{figure}[!h]
\centering
\begin{tikzpicture}[line cap=round,line join=round,>=triangle 45,x=1cm,y=1cm]
\clip(-5, -2.2) rectangle (5,2.2);
\draw [rotate around={0:(0,0)},line width=1pt] (0,0) ellipse (3.6055512754639882cm and 2cm);
\draw [->,line width=1pt] (-1.3064833284084274,1.8640816516113508) -- (-0.388629268518376,1.4671636188346195);
\draw [line width=1pt] (-1.3064833284084274,1.8640816516113508)-- (3.5777138117424787,-0.2480470682797591);
\begin{scriptsize}
\draw [fill=black] (-3,0) circle (2.5pt);
\draw[color=black] (-3,0.4) node {$F_1$};
\draw [fill=black] (3,0) circle (2.5pt);
\draw[color=black] (3,0.4) node {$F_2$};
\draw [fill=black] (-1.3064833284084274,1.8640816516113508) circle (0.5pt);
\draw[color=black] (-1.4594996136510527,2.1) node {$M$};
\draw[color=black] (-0.618040352164444,1.8397384958624228) node {$v$};
\end{scriptsize}
\end{tikzpicture}
\caption{The non-isotropic line passing through $M$ and directed by $v$ goes through a real focus if and only if $P(M,v)=a^{-2}$, see Proposition \ref{prop_lien_invariant_foyers}.}
\end{figure}
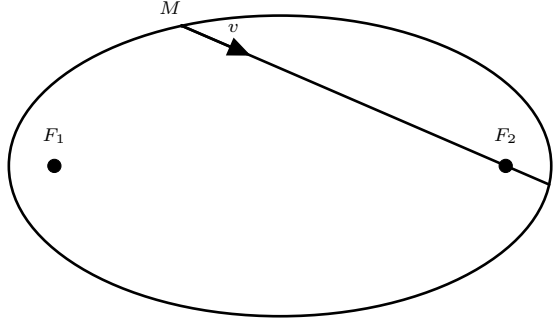

\begin{proof}
We just prove the first point, the second one is analogous.

First, for a fixed $M\in\ellipse$, and a $k\in\cmplx$, there are at most two directions $v$ such that
$$P(M,v) = k.$$
(two collinear vectors have the same direction). Indeed, the equation $\frac{xv_x}{a^2}+\frac{yv_y}{b^2}=k'$  of unknowns $v_x,v_y$ defines a complex line which intersects the affine set $v_x^2+v_y^2=1$
 in at most two points (weak form of Bezout theorem). And considering the same equation but with $-k'$ instead of $k'$, we get at most four unitary vectors such that $P(M,v) = k$, two of them being the opposite of the others. Hence there are at most two directions such that $P(M,v) = a^{-2}$.

We give here the different possibilities for those directions. The real foci of $\ellipse$ have coordinates $(\pm c,0)$ where $c=\sqrt{a^2-b^2}$. Hence $v_{\pm} = (x\pm c,y)$ are directing the lines going through $M$ and the real foci.
Then we have
$$\frac{xv_{+,x}}{a^2}+\frac{yv_{+,y}}{b^2} = \frac{x(x+c)}{a^2}+\frac{y^2}{b^2} = 1+\frac{xc}{a^2} = \frac{a^2+xc}{a^2}$$
and using the fact that $M\in\ellipse$, we have
$$\begin{array}{rcll}
a^2q(v_+) &=& a^2(x+c)^2+a^2y^2\\
 &=& a^2(x+c)^2 +(ab)^2-b^2x^2\\
 &=& c^2x^2+2a^2xc +a^4\\
 &=& (a^2+cx)^2\\
\end{array}$$
Hence, since $M$ is not an isotropic tangency point of $\ellipse$ and by \eqref{equ_tang_points}, we get that
$$q(v_+) \neq0.$$
Thus,
$$P(M,v_+) = a^{-2}$$
and the same is true with $v_-$. There is one case when  $v_-$ and $v_+$ are colinear : when $M$ is one vertex of the ellipse. But this case can be solved easily.
\end{proof}

\begin{corollary}[Forbidden values of $P(T)$]
\label{cor_valeurs_interdites}
Let $T=(M_0,\ldots,M_n)$, with $n\geq3$, be a non-degenerate and non-isotropic piece of trajectory. If $T$ has none of its sides passing through a real or a complex focus of $\ellipse$, then  $P(T) \neq a^{-2}$ and $P(T)\neq b^{-2}$.
\end{corollary}

\begin{figure}[!h]
\centering
\begin{tikzpicture}[line cap=round,line join=round,>=triangle 45,x=0.5cm,y=0.5cm]
\clip(-4.5,-2.6) rectangle (4.5,2.8);
\draw [rotate around={0:(0,0)},line width=1pt] (0,0) ellipse (1.8028cm and 1cm);
\draw [line width=1pt] (-3.6055512754639882,0)-- (3.6055512754639882,0);
\begin{scriptsize}
\draw[color=black] (0,2.5) node {$P(T) = a^{-2}$};
\draw [fill=black] (-3,0) circle (2.5pt);
\draw[color=black] (-3,0.4) node {$F_1$};
\draw [fill=black] (3,0) circle (2.5pt);
\draw[color=black] (3,0.4) node {$F_2$};
\end{scriptsize}
\end{tikzpicture}
\hspace*{1cm}
\begin{tikzpicture}[line cap=round,line join=round,>=triangle 45,x=0.5cm,y=0.5cm]
\clip(-4.5,-2.6) rectangle (4.5,2.8);
\draw [rotate around={0:(0,0)},line width=1pt] (0,0) ellipse (1.8028cm and 1cm);
\draw [line width=1pt] (-3.6055512754639882,0)-- (0,2);
\draw [line width=1pt] (3.6055512754639882,0)-- (0,2);
\draw [line width=1pt] (0,-2)-- (3.6055512754639882,0);
\draw [line width=1pt] (-3.6055512754639882,0)-- (0,-2);
\begin{scriptsize}
\draw[color=black] (0,2.5) node {$P(T) \neq a^{-2},b^{-2}$};
\draw [fill=black] (-3,0) circle (2.5pt);
\draw[color=black] (-2.5,0) node {$F_1$};
\draw [fill=black] (3,0) circle (2.5pt);
\draw[color=black] (2.5,0) node {$F_2$};
\end{scriptsize}
\end{tikzpicture}
\caption{Trajectories with their respectives $P(T)$}
\label{fig_coroll_foyers}
\end{figure}
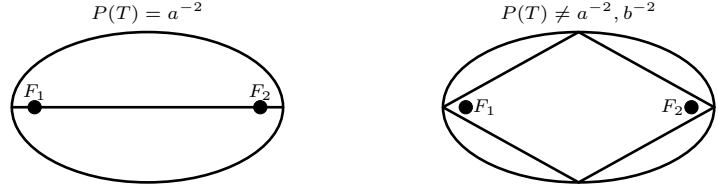

\subsection{Proof of Theorem \ref{main_theorem_caustics}}
\label{subsec_caustics}

Here we prove that the invariant $P(T)$ charaterizes pieces of trajectories which are tangent to the same conic. We first recall the following elementary fact:

\begin{lemma}\label{lemma_conique_duale}
let $C$ be a conic in $\cp^2$ given by the equation $\transp X A X=0$ where $A$ is a $3\times3$ invertible matrix, and $v=(a',b',c')\in\cmplx^3$ defining the line $\ell_v$ of equation $a'x+b'y+c'z=0$. Then $\ell_v$ is tangent to $C$ if and only if
$$\transp v A\ante v = 0.$$
\end{lemma}

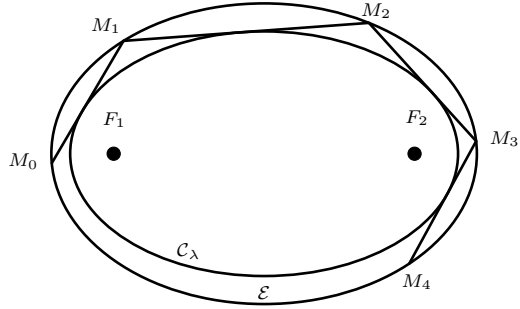
\begin{figure}[!h]
\begin{tikzpicture}[line cap=round,line join=round,>=triangle 45,x=1cm,y=1cm]
\clip(-5,-2.1) rectangle (5,2.1);
\draw [rotate around={0:(0,0)},line width=1pt] (0,0) ellipse (2.8284271247461907cm and 2cm);
\draw [rotate around={0:(0,0)},line width=1pt] (0,0) ellipse (2.57838050536101cm and 1.627281791954208cm);
\draw [line width=1pt] (-2.822840787497323,-0.12563814795175854)-- (-1.869312067799702,1.500945101124356);
\draw [line width=1pt] (-1.869312067799702,1.500945101124356)-- (1.393051102240205,1.740604582687773);
\draw [line width=1pt] (1.393051102240205,1.740604582687773)-- (2.818498874725576,0.1674277354124261);
\draw [line width=1pt] (2.818498874725576,0.1674277354124261)-- (1.928469151211218,-1.463045920814977);
\begin{scriptsize}
\draw [fill=black] (-2,0) circle (2.5pt);
\draw[color=black] (-1.9906447176999233,0.4527745617943184) node {$F_1$};
\draw [fill=black] (2,0) circle (2.5pt);
\draw[color=black] (2.030966426407765,0.4773715718500229) node {$F_2$};
\draw[color=black] (-2.1013312629505934,1.7) node {$M_1$};
\draw[color=black] (1.4898322051822654,1.9) node {$M_2$};
\draw[color=black] (3.2,0.2) node {$M_3$};
\draw[color=black] (-3.2,-0.1) node {$M_0$};
\draw[color=black] (2.030966426407765,-1.7) node {$M_4$};
\draw[color=black] (-1,-1.3) node {$\class_{\lambda}$};
\draw[color=black] (0,-1.85) node {$\ellipse$};
\end{scriptsize}
\end{tikzpicture}
\caption{The confocal caustic $\class_{\lambda}$ inscribed in a piece of billiard trajectory.}
\label{fig_thm_caustic}
\end{figure}

We are now ready to prove the first part of Theorem \ref{main_theorem_caustics}: 

\begin{proposition}
\label{prop_ellipse_inscrite}
Let $T=(M_0,\ldots,M_n)$ be a non-degenerate non-isotropic piece of billiard trajectory. We suppose that none of its sides pass through a (real or complex) focus. Then there is a unique $\lambda\in\cmplx$ such that $a^2+\lambda\neq0$, $b^2+\lambda\neq0$ and $T$ is circumscribed about the conic $\class_{\lambda}$.

We have in this case $\lambda = -(ab)^2P(T)$.
\end{proposition}

\begin{proof} For $s\in\cmplx$, let us define the matrix
$$B_s=\left(\begin{matrix}a^2+s&0&0\\0&b^2+s&0\\0&0&-1 \end{matrix}\right).$$ 
Since the orbit is non-isotropic, two consecutive sides cannot be infinite at the same time. Hence we suppose without loss of generality that $M_0$ is finite. Then the line $M_0M_1$ is defined by the equation in $\cp^2$ $v_yx-v_xy+(v_xy_0-v_yx_0)z=0$. Here $M_0=(x_0,y_0)$ and $v=(v_x,v_y)$ is a directing vector of $M_0M_1$ in $\cmplx^2$. Hence we have $M_0M_1 = \ell_w$ (in the notations of Lemma \ref{lemma_conique_duale}) where
$$w = (v_y,-v_x,v_xy_0-v_yx_0).$$
It allows us to compute
$$\transp wB_sw = (a^2+s)v_y^2+(b^2+s)v_x^2-\left(v_xy_0-v_yx_0\right)^2.$$
Using the fact that $M_0$ lies on the ellipse, that is, substituting 
$$a^2=x_0^2+\frac{a^2y_0^2}{b^2}, \qquad b^2=y_0^2+\frac{b^2x_0^2}{a^2}$$
to the above formula, we get
$$\transp wB_sw = sq(v)+(ab)^2\left(\frac{x_0v_x}{a^2}+\frac{y_0v_y}{b^2}\right)^2.$$
Hence the previous quantity $P(M_0,v)$ defined in Proposition \ref{prop_invariant} appears here again, since
$$\transp wB_sw = q(v)\left(s+(ab)^2P(M_0,v)\right).$$
Hence if $\lambda=-(ab)^2P(M_0,v)$, it is the only $\lambda$ for which $\transp wB_{\lambda}w=0$. And since $P(M_0,v)=P(T)$ doesn't depend on the choice of the index $j$ of a $M_j$, the same computations are true for all lines $M_jM_{j+1}$.
Thus, since $P(T)\neq a^{-2},b^{-2}$ by corollary \ref{cor_valeurs_interdites}, we have $a^2+\lambda\neq0$ and $b^2+\lambda\neq0$, $B_{\lambda}$ is invertible and $B_{\lambda}\ante$ defines the conic $\class_{\lambda}$. The above equality $\transp wB_{\lambda}w=0$ implies that all $M_jM_{j+1}$ are tangent to $\class_{\lambda}$.
\end{proof}

Now let us prove the second part of Theorem \ref{main_theorem_caustics}. It is a consequence of Proposition \ref{prop_tangency_billiard} which comes later. But first we will need the following

\begin{lemma}
\label{lemma_tangente_constante}
Let $\lambda\in\cmplx$ be such that $a^2+\lambda\neq0$ and $b^2+\lambda\neq0$. Then each line $M_0M_1$ which is tangent to $\class_{\lambda}$, where $M_0\neq M_1\in\ellipse$, is such that
\begin{enumerate}
	\item $M_0$ or $M_1$ is finite; 
	\item the line $M_0M_1$ is non-isotropic;
	\item $\lambda = -(ab)^2P(M_j,v)$\\ where $v$ is directing $M_0M_1$ and $M_j$ is a finite point among $M_0$,$M_1$.
\end{enumerate}
\end{lemma}

\begin{proof}
First notice that this line is non-isotropic because otherwise it would be tangent to $\class_{\lambda}$, hence to $\ellipse$, and we could not have $M_0\neq M_1$.

Furthermore, it is not the infinity line (which is not tangent to $\class_{\lambda}$), hence $M_0$ or $M_1$ is finite. Therefore, the lemma results from the computation analogous to that of the proof of Proposition \ref{prop_ellipse_inscrite} for the computation of $\transp wB_sw$.
\end{proof}

\begin{proposition}
\label{prop_tangency_billiard}
Let $\lambda\in\cmplx$ be such that $a^2+\lambda\neq0$ and $b^2+\lambda\neq0$. Then each $n-$uplet of points $T = (M_0,\ldots,M_n)\in\ellipse^n$, two consecutive points being distinct, such that for all $j\in\{1,\ldots,n-1\}$ the sides $M_{j-1}M_j$ and $M_jM_{j+1}$ realize the two tangent lines to $\class_{\lambda}$ going through $M_j$, is a non-degenerate and non-isotropic piece of billiard trajectory, with $\lambda = -(ab)^2P(T)$, whose sides avoid the foci of $\ellipse$.
\end{proposition}

\begin{proof}
Lemma \ref{lemma_tangente_constante} implies that the sides $M_jM_{j+1}$ of $T$ are non-isotropic and for each $j$, at least one point among $M_j$ or $M_{j+1}$ is finite. 

Furthermore, the quantity $P(M_j,v)$ doesn't depend on the finite point $M_j$ of $T$ or on the vector $v$ directing $M_{j-1}M_j$ or $M_jM_{j+1}$. 

Therefore for each $j$, we have two possibilities by Lemmas \ref{lemma_carac_billiard} and \ref{lemma_carac_billiard_2}: either $M_{j-1}M_j=M_jM_{j+1}$ or both lines are symmetric with respect to $T_{M_j}\ellipse$.

Let us show that the former case is a subcase of the latter case. Indeed, if there is a $j$ such that $M_{j-1}M_j=M_jM_{j+1}$, then by properties of conics $M_{j-1}=M_{j+1}$. This implies that there exists only one tangent line to $\class_{\lambda}$ going through $M_j$. Hence $M_j\in\class_{\lambda}\cap\ellipse$ and $M_{j-1}M_j$ is the tangent line $T_{M_j}\class_{\lambda}$, which is orthogonal to the tangent line $T_{M_j}\ellipse$ by Lemma \ref{lemma_common_points}. Thus $M_{j-1}M_j$ and $M_jM_{j+1}$ are symmetric with respect to $T_{M_j}\ellipse$.

Hence for each $j$ we have a billiard reflection. Finally the sides avoid the foci, since $\lambda\neq -a^2,-b^2$ and by Corollary \ref{cor_valeurs_interdites}. This concludes the proof.
\end{proof}

\section{Proof of theorem \ref{main_theorem_orbits}}
\label{sec_n_orbits}

The finiteness of the number of conics $\gamma^n_j$, which we will call \textit{caustics}, is not difficult to prove. For a fixed integer $n\geq 3$, the set $\mathcal{T}_n$ of non-degenerate $n-$periodic orbits is an open set of an algebraic curve of $\ellipse^n\simeq\left(\cp^1\right)^n$ (otherwise we could find an open set of inital conditions $(M_0,M_1)\in\ellipse^2$ corresponding to $n-$periodic orbits, contradicting the real case). This curve has then a finite number of irreducible components. Now, for a fixed caustic, the set of $n-$periodic orbits circumscribed about it is an irreducible algebraic curve included in $\mathcal{T}_n$ (this follows from the fact that each $n-$periodic orbit is uniquely defined by its initial condition $(\ell,M_0)$, where $\ell=M_0M_1$ is a line through $M_0$ that is tangent to the given caustic, and the space of initial conditions is an elliptic curve, see \cite{poncGH}). But two different caustics cannot have the same set of circumscribed orbits: otherwise their corresponding constant $P(T)$ would be the same (Proposition \ref{prop_ellipse_inscrite}), which is impossible. Hence there is a finite number of caustics: $\gamma_1^n=\mathcal{C}_{\lambda_1},\ldots,\gamma_N^n=\mathcal{C}_{\lambda_N}$, with pairwise distinct $\lambda_j$ all different from $-a^2$ and $-b^2$.

Our goal now is to estimate the number $N$ of caustics. To do so, we will use Cayley's theorem, proven for example in \cite{poncGH}. We will just give an upper bound on $N$ (Propositions \ref{prop_finite_confocal} and \ref{prop:degree_polynomial}) and explain how we can compute its exact value (Corollary \ref{cor:exact_nb_roots} and Propositions \ref{prop:degree_polynomial} and \ref{prop:forbidden_roots}).

The following theorem of Cayley is needed. We will say that two conics $C$ and $D$ are \textit{in general position} if they intersect (transversally) at four points. Note that if $\lambda\neq 0$, $\mathcal{C}_{\lambda}$ and $\ellipse$ are in general position (see Lemma \ref{lemma_common_points}).

\begin{theorem}[Cayley]
\label{thm_cayley}
Fix $n\geq 3$. Let $C$ and $D$ be quadratic forms defining two regular conics in $\cp^2$ in general position. Write 
$$\sqrt{\det(tC+D)} = A_0+A_1t+A_2t^2\ldots$$
the analytic expansion in $0$ of the holomorphic function $t\mapsto\sqrt{\det(tC+D)}$. Then there is an $n-$sided polygon inscribed in $C$ and circumscribed about $D$ if and only if
$$\left|\begin{matrix}
A_2&\ldots&A_{m+1}\\
\vdots & \ddots & \vdots\\
A_{m+1}&\ldots&A_{2m}
\end{matrix}\right|=0, \qquad \text{with } m=\frac{n-1}{2} \;\; (n\text{ odd}),$$
or
$$\left|\begin{matrix}
A_3&\ldots&A_{m+1}\\
\vdots & \ddots & \vdots\\
A_{m+1}&\ldots&A_{2m-1}
\end{matrix}\right|=0 \qquad \text{with } m=\frac{n}{2} \;\; (n\text{ even}).$$
This condition is reduced to $A_2=0$ when $n=3$ and to $A_3=0$ when $n=4$.
\end{theorem}

\begin{remark}
Note that the determinants we are considering in Theorem \ref{thm_cayley} can be rewritten for $n$ odd and $n$ even respectively as
$$\det (A_{i+j})_{1\leq i,j\leq m}\qquad\text{ and }\qquad\det (A_{i+j+1})_{1\leq i,j\leq m-1}.$$
\end{remark}

\begin{remark}
\label{rem_hom_coord}
We recall that an element of $\cp^2$ can be represented by a triple of the form $(x:y:z)$ where $x,y,z\in\cmplx$ are not all equal to $0$. Be careful that this representation is not unique since $(x:y:z)=(tx:ty:tz)$ for any $t\in\cmplx\etoile$. Then any polynomial $P(x,y)$ of degree $d$ can be associated to a homogeneous polynomial 
$$P^\sharp(x,y,z)=z^dP(\frac{x}{z},\frac{y}{z}).$$
Hence the zeros of $P$ in $\cmplx^2$ can be extended in $\cp^2$ to the set of zeros of $P^\sharp(x,y,z)$. In our case, the conic $\mathcal{C}_{\lambda}$ can be viewed in $\cp^2$ as the set of $(x:y:z)$ such that
$$\frac{x^2}{a^2+\lambda}+\frac{y^2}{b^2+\lambda}-z^2=0.$$
\end{remark}

\begin{proposition}
\label{prop_finite_confocal}
Let $n\geq3$. There is a polynomial $\mathcal{B}^n(\lambda)$ such that any of its roots $\lambda\notin\{-a^2,-b^2\}$ has the following property: there exists an $n-$sided polygon inscribed in $\ellipse$ and circumscribed about $\class_{\lambda}$. 

The degree of $\mathcal{B}^n$ is such that 
$$\deg\mathcal{B}^n\leq\left\{\begin{matrix}
\frac{n^2-1}{4}&\text{if } n \text{ is odd}\\
\frac{n^2}{4}&\text{if } n \text{ is even.}\\
\end{matrix}\right.$$
\end{proposition}

\begin{proof}
Suppose first that $n=2m+1$ is odd and fix a $\lambda$ with $\lambda+a^2,\lambda+b^2\neq0$. To understand if there is an $n-$sided polygon inscribed in $\ellipse$ and circumscribed about $\class_{\lambda}$, we apply Cayley's theorem: there is such a polygon if and only if the determinant
$$\mathcal{A}^n(\lambda) = \left|\begin{matrix}
A_2(\lambda)&\ldots&A_{m+1}(\lambda)\\
\vdots & \ddots & \vdots\\
A_{m+1}(\lambda)&\ldots&A_{2m}(\lambda)
\end{matrix}\right|$$
vanishes, where the $A_i(\lambda)$ are the coefficients in the analytic expansion of 
$$f:t\to\sqrt{\det(tC+D_{\lambda})}$$
with $C$ and $D_{\lambda}$ being quadratic forms respectively associated to $\ellipse$ and to $\class_{\lambda}$. Thus, to prove the result we want, we have to show that the determinant $\mathcal{A}^n(\lambda)$ vanishes for a finite number of $\lambda$. Let us give a more precise fomula of $\mathcal{A}^n(\lambda)$. We have by Remark \ref{rem_hom_coord},
$$tC+D_{\lambda} = \left(\frac{t}{a^2}+\frac{1}{a^2+\lambda}\right)x^2+\left(\frac{t}{b^2}+\frac{1}{b^2+\lambda}\right)y^2-(t+1)z^2$$
hence
$$\det(tC+D_{\lambda}) =-\left(\frac{t}{a^2}+\frac{1}{a^2+\lambda}\right)\left(\frac{t}{b^2}+\frac{1}{b^2+\lambda}\right)(t+1)$$
which we factorize in
$$\det(tC+D_{\lambda}) =-\frac{1}{(a^2+\lambda)(b^2+\lambda)}\left(\frac{a^2+\lambda}{a^2}t+1\right)\left(\frac{b^2+\lambda}{b^2}t+1\right)(t+1).$$
Define the map $g:t\mapsto\sqrt{\left(\frac{a^2+\lambda}{a^2}t+1\right)\left(\frac{b^2+\lambda}{b^2}t+1\right)(t+1)}$ and write its Taylor expansion as
$$g(t) = \sum_{k=0}^{\infty} B_k(\lambda) t^k$$
Since 
$$f(t) = \frac{ig(t)}{\sqrt{(a^2+\lambda)(b^2+\lambda)}}$$
we have
$$A_k(\lambda) =\frac{iB_k(\lambda)}{\sqrt{(a^2+\lambda)(b^2+\lambda)}}.$$
This shows that $\mathcal{A}^n(\lambda)$ is a function of $\lambda$ which vanishes if and only if the determinant 
$$\mathcal{B}^n(\lambda) = \left|\begin{matrix}
B_2(\lambda)&\ldots&B_{m+1}(\lambda)\\
\vdots & \ddots & \vdots\\
B_{m+1}(\lambda)&\ldots&B_{2m}(\lambda)
\end{matrix}\right|$$
also vanishes. We thus need to compute the $B_k$'s. Write
$$\sqrt{t+1} = c_0+c_1t+c_2t^2+\ldots$$
where
\begin{equation}
\label{eq:c_k}
c_k = \frac{1}{k!}\left(\frac{1}{2}\right)\left(\frac{1}{2}-1\right)\ldots\left(\frac{1}{2}-k+1\right) = \frac{(-1)^{k+1}}{4^k(2k-1)}\binom{2k}{k}.
\end{equation}
Therefore for any $\beta$, 
$$\sqrt{\beta t+1} = c_0+c_1\beta t+c_2\beta^2 t^2+\ldots$$
where
\begin{equation}
\label{eq:B_k}
B_k(\lambda) = \sum_{u+v+w=k} \frac{c_uc_vc_w}{a^{2u}b^{2v}}(a^2+\lambda)^u(b^2+\lambda)^v.
\end{equation}
Therefore each $B_k$ is a polynomial in $\lambda$ of degree at least $k$. Hence $\mathcal{B}^n(\lambda)$ is a polynomial in $\lambda$ verifying: for any $\lambda\notin\{-a^2,-b^2\}$, $\mathcal{B}^n(\lambda)=0$ if and only if $\mathcal{A}^n(\lambda)=0$, which is true if and only if there exists an $n-$sided polygon inscribed in $\ellipse$ and circumscribed about $\class_{\lambda}$. 

Now for any permutation $\sigma$ of $\{1,\ldots,m\}$ we have
$$\deg \prod_{j=1}^{m} B_{\sigma(j)+j} = \sum_{j=1}^m \deg B_{\sigma(j)+j} \leq \sum_{j=1}^m \left(\sigma(j)+j\right) = m(m+1)$$
and since $\mathcal{B}^n(\lambda)$ is a sum of $\pm\prod_{j=1}^{m} B_{\sigma(j)+j}$ over all $\sigma$, we have that $\deg \mathcal{B}^n(\lambda) \leq m(m+1) = \frac{n^2-1}{4}$.

Now if $n=2m$ is even, the existence of $\mathcal{B}^n$ is treated exactly as in the case when $n$ is odd, but instead, $\mathcal{B}^n(\lambda)=\det(B_{i+j+1})_{1\leq i,j\leq m-1}$. Hence for any permutation $\sigma$ of $\{1,\ldots,m\}$ we have
$$\deg \prod_{j=1}^{m-1} B_{\sigma(j)+j+1} = m^2$$
and $\deg\mathcal{B}^n\leq m^2=\frac{n^2}{4}$.
\end{proof}

Now we are ready for the proof of Theorem \ref{main_theorem_orbits}:

\begin{proof}[Proof of Theorem \ref{main_theorem_orbits} (without its last statement for $n=3, 4$)]
Let $\mathcal{B}^n$ be the polynomial of Proposition \ref{prop_finite_confocal}. By construction, any root $\lambda$ of $\mathcal{B}^n$ different from $-a^2$ and $-b^2$ corresponds to a $\mathcal{C}_{\lambda}$ inscribed in an $n$-sided polygon $P$ of $\ellipse$. By Theorem \ref{main_theorem_caustics}, $P$ is a $n$-periodic billard orbit and $\mathcal{C}_{\lambda}$ is its caustic, therefore $\lambda=\lambda_j$ for a certain $1\leq j\leq N$. Conversely, all $\lambda_j$ are roots of $\mathcal{B}^n$ since a periodic billard orbit is a polygon.

The first statement of Theorem \ref{main_theorem_orbits} is obvious by theorem \ref{main_theorem_caustics}. The second statement is a consequence of the argument given at the beginning of Section \ref{sec_n_orbits}.

The third statement comes from Poncelet's theorem: by definition of the $\gamma_j$, there is at least one $n-$sided polygon inscribed in $\ellipse$ and circumscribed about $\class_{\lambda}$. Now Poncelet's theorem states that for any $p\in\ellipse$ there exists such a polygon with $p$ as a vertex.
\end{proof}

We can try to know if there are periodic orbits passing through the foci. We recall that in the real case, a billard trajectory going through the real foci accumulates on the (real) foci line.

\begin{proposition}
\label{prop:periodic_foci}
If an $n$-periodic orbit goes through the real (resp. complex) foci, then $n$ is even and its edges are on the real (resp. complex) foci line.
\end{proposition}

\begin{proof}
Since a line through a focus is reflected into a line through another focus, $n$ has to be even. Now consider the map $f$ from $\ellipse$ to $\ellipse$ defined as follows: for $A$ in $\ellipse$, let $B$ be the other intersection point of $AF_1$ with $\ellipse$ and $C$ the other intersection point of $BF_2$ with $\ellipse$, where $F_1,F_2$ are the real foci of $\ellipse$. A $n$-periodic orbit through the real foci has a vertex $M$ such that $f^n(M)=M$. Now $f$ is a non trivial automorphism of $\ellipse\simeq\cp^1$ hence is a Möbius transform, and so is $f^n$. But such tranform has at least two fixed points. Since we already know two of them to be the vertices of $\ellipse$, the latter correspond to the only possible initial points of periodic orbits through the real foci. The same holds with complex foci.
\end{proof}

From the proof of Theorem \ref{main_theorem_orbits}, we get the

\begin{corollary}
\label{cor:exact_nb_roots}
If the roots of $\mathcal{B}^n$ are simple and different from $-a^2$ and $-b^{2}$, then $N=\deg \mathcal{B}^n$.
\end{corollary}

\begin{remark}
In order to compute the exact value of $N$, we can first try to compute the exact value of $\deg \mathcal{B}^n$ (see Proposition \ref{prop:degree_polynomial}). Then we can understand when the assumptions of Corollary \ref{cor:exact_nb_roots} are satisfied (see Proposition \ref{prop:forbidden_roots}). Unfortunately we still do not know in which cases $\mathcal{B}^n$ has simple roots.
\end{remark}

\begin{proposition}
\label{prop:degree_polynomial}
There exist $r_1,\ldots,r_p\in\real^+$ such that for all $(a,b)$ with $a\notin\{r_1 b,\ldots,r_pb\}$, we have 
$$\deg\mathcal{B}^n=\left\{\begin{matrix}
\frac{n^2-1}{4}&\text{if } n \text{ is odd}\\
\frac{n^2}{4}&\text{if } n \text{ is even.}\\
\end{matrix}\right.$$
\end{proposition}

\begin{remark}
We show in Proposition \ref{prop:circle_polynomial} that $p\geq 1$, more precisely that $1$ belongs to the collection of $\{r_1,\ldots,r_p\}$.
\end{remark}

\begin{proof}
Suppose $n=2m+1$ is odd. By Equation \eqref{eq:B_k}, $B_k$ is of degree $\leq k$ and the coefficient in front of $\lambda^k$ is
\begin{equation}
\label{eq:domin_B_k}
d(B_k)=\sum_{u+v=k} \frac{c_uc_v}{a^{2u}b^{2v}}=\frac{(-1)^k}{4^{k}}\sum_{u+v=k} \frac{1}{a^{2u}b^{2v}(2u-1)(2v-1)}\binom{2u}{u}\binom{2v}{v}.
\end{equation}
Fix a permutation $\sigma$ of $\{1,\ldots,m\}$. We have
$$\deg \prod_{j=1}^{m} B_{\sigma(j)+j} = \sum_{j=1}^m \deg B_{\sigma(j)+j} \leq \sum_{j=1}^m \left(\sigma(j)+j\right) = m(m+1)$$
and the coefficient in front of $\lambda^{m(m+1)}$ is $\prod_{j=1}^{m} d(B_{\sigma(j)+j})$.
Since $\mathcal{B}^n(\lambda)$ is a sum of $\pm\prod_{j=1}^{m} B_{\sigma(j)+j}$ over all $\sigma$, we have that $\deg \mathcal{B}^n(\lambda) \leq m(m+1)$, and the coefficient in front of $\lambda^{m(m+1)}$ is
$$d_n(a,b)=\left|\begin{matrix}
d(B_2)&\ldots&d(B_{m+1})\\
\vdots & \ddots & \vdots\\
d(B_{m+1})&\ldots&d(B_{2m})
\end{matrix}\right|.$$
Let us show that $d_n(a,b)\neq 0$ except for specific $(a,b)$ as described in Proposition \ref{prop:degree_polynomial}. Note first that each $d(B_k)$ is a homogeneous polynomial in $(a^{-1},b^{-1})$ of degree $2k$, and by Equation \eqref{eq:domin_B_k} the coefficient in front of $a^{-2k}$ is 
$$\frac{(-1)^{k+1}}{4^{k}(2k-1)}\binom{2k}{k}=2\frac{(-1)^{k+1}}{4^{k}}\cat_{k-1}$$
where $\cat_k=\frac{1}{k+1}\binom{2k}{k}$ is the $k$-th Catalan number.

Now by linearity of the determinant, $d_n(a,b)$ is also a homogeneous polynomial in $(a^{-1},b^{-1})$, and we apply the same procedure as before: for any permutation $\sigma$ of $\{1,\ldots,m\}$, we have 
$$\deg \prod_{j=1}^{m} d(B_{\sigma(j)+j}) = \sum_{j=1}^m \deg d(B_{\sigma(j)+j}) = \sum_{j=1}^m 2\left(\sigma(j)+j\right) = 2m(m+1)$$
and the coefficient in front of $a^{-2k}$ is
$$\prod_{j=1}^{m} 2\frac{(-1)^{j+\sigma(j)+1}}{4^{j+\sigma(j)}}\cat_{j+\sigma(j)-1}=\frac{(-1)^{m(m+2)}2^m}{4^{m(m+1)}}\prod_{j=1}^{m}\cat_{j+\sigma(j)-1}.$$
Since $d_n(a,b)$ is a sum of $\pm\prod_{j=1}^{m} d(B_{\sigma(j)+j})$ over all $\sigma$, we have that $\deg d_n(a,b) \leq 2m(m+1)$, and the coefficient in front of $a^{-2m(m+1)}$ is
$$\frac{(-1)^{m}}{2^{m(2m+1)}}\det H_m$$
where $H_m$ is a Hankel matrix of the sequence $(\cat_{k+1})_k$ defined as
$$H_m = \left(\begin{matrix}
\cat_1&\cat_2&\cdots&\cat_m\\
\cat_2&\cat_3&&\\
\vdots&&\ddots&\vdots\\
\cat_m&&\cdots&\cat_{2m}
\end{matrix}\right).$$
There are not so easy methods to show that $\det H_m=1$, see for example \cite{Krattenthaler}, Theorem $33$. Hence $d_n(a,b)$ is a nonzero homogeneous polynomial in $(a^{-1},b^{-1})$ and therefore there exists a at most finite collection of numbers $r_1,\ldots,r_p\in\real^+$ such that for all $a,b>0$, we have $d_n(a,b)=0$ if and only if $a\in\{r_1b,\ldots,r_pb\}$.
\end{proof}

We can ask the question if $\deg\mathcal{B}^n$ has always the maximal value (described in Proposition \ref{prop:degree_polynomial}), or if we can find $a,b>0$ such that $\deg\mathcal{B}^n$ is less than the value in Proposition \ref{prop:degree_polynomial}. Proposition \ref{prop:circle_polynomial} asserts that we can find indeed such $a,b$ by looking at the case of the circle ($a=b$).

\begin{proposition}
\label{prop:circle_polynomial}
When $a=b$ (in the case of the circle), 
$$\deg\mathcal{B}^n=\left\{\begin{array}{cl}
\frac{n-1}{2}&\text{if } n \text{ is odd}\\
\frac{n}{2}-1&\text{if } n \text{ is even.}\\
\end{array}\right.$$
\end{proposition}

\begin{proof}
Suppose $n=2m+1$ is odd. By Equation \ref{eq:B_k}, when $a=b=R$, for $k\geq2$
$$B_k = \sum_{w=0}^k c_{k-w}\left(1+\frac{\lambda}{a^2}\right)^w\sum_{u+v=w}c_uc_v.$$
Let us compute $\sum_{u+v=w}c_uc_v$: it is the Taylor coefficient at $t^w$ of the function $\sqrt{1+t}^2=1+t$, therefore we get that 
$$\sum_{u+v=w}c_uc_v=\left\{\begin{array}{cl}
1&\text{if } 0\leq w\leq 1\\
0&\text{if } w\geq 2.
\end{array}\right.$$
Hence $B_k = c_k+c_{k-1}x$ where $x = 1+\lambda/a^2$. Using the multilinearity of $\det$, it is not hard to see that $\mathcal{B}^n$ is of degree $m$ if $n$ is odd and $m-1$ if $n$ is even.
\end{proof}

\begin{proposition}
\label{prop:forbidden_roots}
There exist a at most finite collection of numbers $r'_1,\ldots,r'_q\in\real^+$ such that for all $(a,b)$ with $a\notin\{r'_1 b,\ldots,r'_qb\}$, $-a^2$ and $-b^2$ are not roots of $\mathcal{B}^n$.
\end{proposition}

\begin{proof}
Suppose $n=2m+1$ is odd. By Equation \eqref{eq:B_k}, for $k\geq 2$,
\begin{equation}
\label{eq:domin_B_k}
B_k\left(-a^2\right)=\sum_{v+w=k} \frac{c_vc_w}{b^{2v}}(b^2-a^2)^v=\frac{1}{b^{2k}}\sum_{v+w=k} c_vc_wb^{2w}(b^2-a^2)^v = \frac{1}{b^{2k}}P_k(a,b)
\end{equation}
where $P_k(a,b)$ is a homogeneous polynomial in $(a,b)$ of degree $2k$. The coefficient in front of $a^{2k}$ is 
$$(-1)^kc_k=-\frac{1}{4^k(2k-1)}\binom{2k}{k}=-\frac{\cat_{k-1}}{2^{2k-1}}.$$
As in the proof of Proposition \ref{prop:degree_polynomial}, for any permutation $\sigma$ of $\{1,\ldots,m\}$,
$$\prod_{j=1}^{m} B_{\sigma(j)+j}(-a^2)=\prod_{j=1}^{m} \frac{1}{b^{2(\sigma(j)+j)}}P_{\sigma(j)+j}(a,b)=\frac{Q_{\sigma}(a,b)}{b^{2m(m+1)}}$$
where $Q_{\sigma}(a,b)$ is a homogeneous polynomial of degree 
$$\sum_{j=1}^m \deg P_{j+\sigma(j)}=\sum_{j=1}^m 2(\sigma(j)+j) = 2m(m+1)$$
whose coefficient in front of $a^{2m(m+1)}$ is
$$\prod_{j=1}^{m}\left(-\frac{\cat_{j+\sigma(j)-1}}{2^{2(j+\sigma(j))-1}}\right)=\frac{(-1)^m}{2^{m(2m+1)}}\prod_{j=1}^{m}\cat_{j+\sigma(j)-1}.$$
As in proof of Proposition \ref{prop:degree_polynomial}, $\mathcal{B}^n(-a^2)$ is a sum of products of the form $\pm\prod_{j=1}^{m} B_{\sigma(j)+j}(-a^2)$ hence can be written as 
$$\frac{R_n(a,b)}{b^{2m(m+1)}}$$
where $R_n(a,b)$ is the sum of $\varepsilon(\sigma)\prod_{j=1}^{m} Q_{\sigma(j)+j}(a,b)$ and $\varepsilon(\sigma)$ is the parity of $\sigma$. Thus $R_n(a,b)$ is a homogeneous polynomial of degree $2m(m+1)$ whose coefficient in front of $a^{2m(m+1)}$ is 
$$\frac{(-1)^m}{2^{m(2m+1)}}\det H_m = \frac{(-1)^m}{2^{m(2m+1)}}$$ as in the proof of Proposition \ref{prop:degree_polynomial}. Thus $R_n(a,b)$ is a nonzero homogeneous polynomial such that
$$\mathcal{B}^n(-a^2)=\frac{R_n(a,b)}{b^{2m(m+1)}}.$$
We can do the same with $\mathcal{B}^n(-b^2)$ to obtain the same conclusion, which finishes the proof.
\end{proof}


\section{Triangular orbits}
\label{sec_3_orbits}

As in the proof of Theorem \ref{main_theorem_orbits}, any root $\lambda\notin\{-a^2,-b^2\}$ of $\mathcal{B}^3(\lambda)$ corresponds to a caustic defined by $\mathcal{C}_{\lambda}$ for triangular orbits in the complex domain. Therefore we are going to compute $\mathcal{B}^3(\lambda)$ and its roots. This section is very similar to \cite{DragRad2019} p. 17, since we get the same results with a similar method (just notice that the conventions adopted in this paper is different from ours: in \cite{DragRad2019}, $a$ stand for $a^2$, $b$ for $b^2$ and their $\mathcal{C}_{\lambda}$ is the same as our $\mathcal{C}_{-\lambda}$). Yet, our conclusion is a little bit more general thanks to our previous study on complex caustics (Theorem \ref{main_theorem_caustics}). We have
$$\mathcal{B}^3(\lambda) = -\frac{1}{8a^4b^4}\left((a^2 - b^2)^2\lambda^2 - 2(a^4b^2 + a^2b^4)\lambda-3a^4b^4\right)$$
which is a polynomial of second degree. Its roots are
$$\lambda_{\pm} = \frac{a^2b^2(a^2+b^2)\pm 2a^2b^2\sqrt{a^4 - a^2b^2 + b^4}}{(a^2 - b^2)^2}$$
which correspond to the opposite of the solutions found in \cite{DragRad2019} as predicted.

\begin{lemma}
\label{lemma_param_caustics_3}
We have $\lambda_{+}>0$ and $-b^2<\lambda_{-}<0$.
\end{lemma}

\begin{proof}
The inequality $\lambda_{-}>-b^2$ is equivalent to 
$$2a^2b^2\sqrt{a^4-2a^2b^2+b^4}-a^2b^2(a^2+b^2)<b^2(a^2-b^2)^2$$
which can be simplified in 
$$2a^2b^2\sqrt{a^4-2a^2b^2+b^4}<2a^4+b^4-a^2b^2.$$
This last inequality is true as we can see by taking the square of both its right and left sides. The same methode can be applied to show that $\lambda_{-}<0$.
\end{proof}

\noindent By Lemma \ref{lemma_param_caustics_3}, the caustics corresponding to $\lambda_{\pm}$, $\mathcal{C}_{\lambda_{+}}$ and $\mathcal{C}_{\lambda_{-}}$, are confocal ellipses which are respectively bigger and smaller than $\ellipse =\mathcal{C}_{0}$. As noticed in \cite{DragRad2019}, $\mathcal{C}_{\lambda_-}$ corresponds to the real motion. And we can add that $\mathcal{C}_{\lambda_+}$ corresponds to no real orbits by convexity. Hence we deduce the following corollary, which is a version of Theorem \ref{main_theorem_orbits} in the case when $n=3$:

\begin{corollary}
\label{cor_deux_caustiques}
There exist two distinct real ellipses $\gamma^3_1 := \class_{\lambda_-}$ and $\gamma^3_2 := \class_{\lambda_+}$ which are confocal to $\ellipse$ by construction, and such that
\begin{itemize}
	\item all complex triangles inscribed in $\ellipse$ and circumscribed about a $\gamma^3_j$ are billiard orbits; 
	\item any complex triangular billiard orbit of $\ellipse$ is circumscribed about either $\gamma^3_1$, or $\gamma^3_2$;
	\item any complex orbit inscribed in $\ellipse$ and circumscribed about a $\gamma^3_j$ is a triangular orbit.
\end{itemize}
\end{corollary}

\begin{example}
We consider the case when $a=2$ and $b=1$, see Figure \ref{fig_2caustics}. We compute that
$$\lambda_- = \frac{20-8\sqrt{13}}{9}\approx-0.9827 \quad\text{ and }\quad \lambda_+ = \frac{20+8\sqrt{13}}{9}\approx5.4272.$$
\end{example}

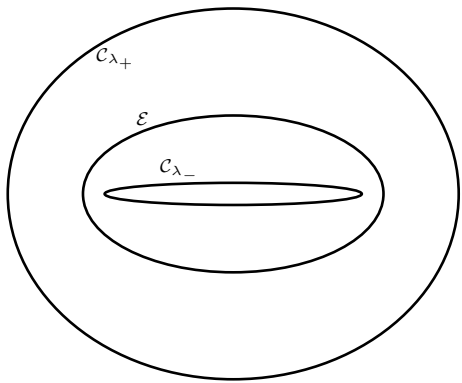
\begin{figure}
\centering
\begin{tikzpicture}[line cap=round,line join=round,>=triangle 45,x=1cm,y=1cm]
\clip(-3.5,-2.7) rectangle (3.5,2.7);
\draw [rotate around={0:(0.0025086145409116416,0)},line width=1pt] (0.0025086145409116416,0) ellipse (1.9974913854590879cm and 1.0432433360133084cm);
\draw [rotate around={0:(0.0025086145409115726,0)},line width=1pt] (0.0025086145409115726,0) ellipse (2.9974913854590883cm and 2.4664426668897765cm);
\draw [rotate around={0:(0.002508614540915727,0)},line width=1pt] (0.002508614540915727,0) ellipse (1.709768500390961cm and 0.1472859398657943cm);
\begin{scriptsize}
\draw[color=black] (-1.2,1) node {$\ellipse$};
\draw[color=black] (-1.5597148670526204,1.8) node {$\mathcal{C}_{\lambda_+}$};
\draw[color=black] (-0.7111058200015711,0.31037283751140887) node {$\mathcal{C}_{\lambda_-}$};
\end{scriptsize}
\end{tikzpicture}
\caption{The initial ellipse with its two caustics $\mathcal{C}_{\lambda_-}$ and $\mathcal{C}_{\lambda_+}$ when $n=3$, $a=2$ and $b=1$.}
\label{fig_2caustics}
\end{figure}

\noindent We can apply Corollary \ref{cor_deux_caustiques} to classify the degenerate triangular orbits:

\begin{proposition}
There are exacty $8$ degenerate triangular orbits, given by an isotropic tangency point $\alpha$ of $\ellipse$ and a point $\beta$ of $\ellipse$ such that $\alpha\beta$ is tangent to $\gamma^3_j$ for some $j=1,2$ and non-isotropic.
\end{proposition}

\begin{proof}
Degenerate triangular orbits are limits of non-degenerate (non-isotropic) triangular orbits, hence are circumscribed about a $\gamma^3_j$. We apply Proposition \ref{lemma_olga}: one side $A$ is isotropic and tangent to a $\gamma^3_j$. This gives only $4$ possible positions. The other side is non-isotropic, which gives two other possibilities ($B$ should be tangent to the same $\gamma_j^3$), once $A$ is fixed.
\end{proof}

\section{$4-$periodic orbits}
\label{sec_4_orbits}

We apply in this section the same ideas as in Section \ref{sec_3_orbits} for $n=4$. We compute $\mathcal{B}^4$ and its roots
$$\mathcal{B}^4(\lambda) = \frac{1}{16a^6b^6}\left((a^2-b^2)^2(a^2+b^2)\lambda^3 + (a^3b-ab^3)^2\lambda^2 - (a^6b^4 + a^4b^6)\lambda-a^6b^6\right).$$
We can check that its roots are
$$\lambda_1 = -\frac{a^2b^2}{a^2-b^2}, \qquad \lambda_2 = -\frac{a^2b^2}{a^2+b^2},\qquad \lambda_3 = \frac{a^2b^2}{a^2-b^2}.$$
They satisfy
\begin{equation}
\label{ineq:lambda_i}
\lambda_1<-b^2<\lambda_2<0<\lambda_3.
\end{equation}
Note also that
\begin{equation}
\label{ineq:lambda_1}
\lambda_1
\left\{\begin{matrix}
&<-a^2&\text{ when } a<\sqrt{2}b\\
&=-a^2&\text{ when } a=\sqrt{2}b\\
&>-a^2&\text{ when } a>\sqrt{2}b\\
\end{matrix}\right.
\end{equation}
Denote by $\mathcal{C}_i$ the caustic $\mathcal{C}_{\lambda_i}$ where $i=1,2,3$, with $\mathcal{C}_i$ not defined when $a=\sqrt{2}b$. By Inequations \eqref{ineq:lambda_i}, $\mathcal{C}_2$ and $\mathcal{C}_3$ are confocal ellipses, respectively smaller and bigger than $\ellipse=\mathcal{C}_0$. By Inequation \eqref{ineq:lambda_1}, the conic $\mathcal{C}_1$ is a hyperbola if and only if $a>\sqrt{2}b$.

\begin{corollary}
\label{cor_trois_caustiques}
When $a\neq\sqrt{2}b$ (resp. when $a=\sqrt{2}b$), there exist three (resp. two) distincts conics $\mathcal{C}_1$, $\mathcal{C}_2$, $\mathcal{C}_3$ (resp. $\mathcal{C}_2$, $\mathcal{C}_3)$ which have the previous described properties, and such that
\begin{itemize}
	\item all complex quadrilaterals inscribed in $\ellipse$ and circumscribed about a $\mathcal{C}_j$ are billiard orbits; 
	\item any complex quadrilateral billiard orbit of $\ellipse$, which do not have its edges on a foci line, is circumscribed about a $\mathcal{C}_j$;
	\item any complex orbit inscribed in $\ellipse$ and circumscribed about a $\mathcal{C}_j$ is a quadrilateral orbit.
\end{itemize}
\end{corollary}

\begin{example}
We consider the case when $a=2$ and $b=1$, see Figure \ref{fig_3caustics}. We compute that
$$\lambda_3=-\lambda_1= \frac{4}{3}\approx1,333\quad\text{ and }\quad \lambda_2 = \frac{4}{5}= 0,8.$$
\end{example}

\begin{figure}
\centering
\begin{tikzpicture}[line cap=round,line join=round,>=triangle 45,x=1cm,y=1cm]
\clip(-3,-1.8) rectangle (3,1.8);
\draw [rotate around={0:(0.004465809016780549,0)},line width=1pt] (0.004465809016780549,0) ellipse (1.9955341909832196cm and 1.0008587655128343cm);
\draw [rotate around={0:(0.004465809016780692,0)},line width=1pt] (0.004465809016780692,0) ellipse (1.784232681737871cm and 0.45060828188388125cm);
\draw [rotate around={0:(0.004465809016780525,0)},line width=1pt] (0.004465809016780525,0) ellipse (2.2943056805364748cm and 1.5110923588129084cm);
\draw [samples=50,domain=-0.99:0.99,rotate around={0:(0.0044658090167803855,0)},xshift=0.0044658090167803855cm,yshift=0cm,line width=1pt] plot ({1.6329479850417183*(1+(\x)^2)/(1-(\x)^2)},{0.5602850319501439*2*(\x)/(1-(\x)^2)});
\draw [samples=50,domain=-0.99:0.99,rotate around={0:(0.0044658090167803855,0)},xshift=0.0044658090167803855cm,yshift=0cm,line width=1pt] plot ({1.6329479850417183*(-1-(\x)^2)/(1-(\x)^2)},{0.5602850319501439*(-2)*(\x)/(1-(\x)^2)});
\begin{scriptsize}
\draw[color=black] (-0.9777216831311774,1.02) node {$\ellipse$};
\draw[color=black] (-0.8726664469978024,0.57) node {$\mathcal{C}_2$};
\draw[color=black] (-1.1272233653209804,1.5) node {$\mathcal{C}_3$};
\draw[color=black] (2.7,0.9599478885188634) node {$\mathcal{C}_1$};
\end{scriptsize}
\end{tikzpicture}
\caption{The initial ellipse with its three caustics in the case when $n=4$, $a=2$ and $b=1$.}
\label{fig_3caustics}
\end{figure}
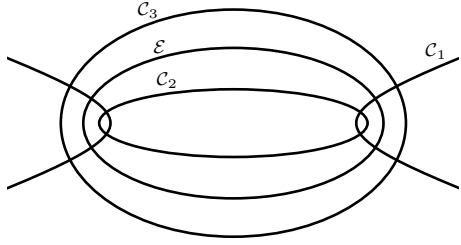

Let us investigate to which type of orbits corresponds each of the $\mathcal{C}_i$. To do so, we fix a point $S$ to be the vertex $(-a,0)$ of the initial ellipse $\ellipse$, and we want to determine the $4$-periodic orbits having $S$ as a vertex. One is already known: it is the \textit{flat} orbit whose vertices are on the foci-line and denoted by $T_0^4$, see Figure \ref{fig_flat_orbit}.

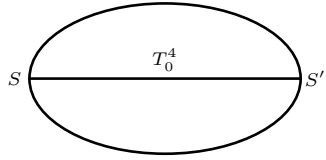
\begin{figure}[!h]
\centering
\begin{tikzpicture}[line cap=round,line join=round,>=triangle 45,x=0.5cm,y=0.5cm]
\clip(-4.5,-2.6) rectangle (4.5,2.6);
\draw [rotate around={0:(0,0)},line width=1pt] (0,0) ellipse (1.8028cm and 1cm);
\draw [line width=1pt] (-3.6055512754639882,0)-- (3.6055512754639882,0);
\begin{scriptsize}
\draw[color=black] (-4,0) node {$S$};
\draw[color=black] (4,0) node {$S'$};
\draw[color=black] (0,0.5) node {$T_0^4$};
\end{scriptsize}
\end{tikzpicture}
\caption{The flat $4$-periodic orbit $T_0^4$}
\label{fig_flat_orbit}
\end{figure}

By Theorem \ref{main_theorem_orbits}, each one of the other orbits is tangent to a $\mathcal{C}_i$: denote by $T^4_i$ the $4$-periodic orbit tangent to $\mathcal{C}_i$, where $i=1,2,3$. Let $v_i=(v_{x,i}, v_{y,i})$ be a vector directing a side of $T_i^4$ starting from $S$, and such that $q(v_i)=1$. By Proposition \ref{prop_tangency_billiard}, one has 
\begin{equation}
\label{equ:v_from_lambda}
\lambda_i=-(ab)^2P(T_i^4) = -b^2v_{x,i}^2.
\end{equation}
Thus we can determine $v_i$ for each $i$.\\

\noindent\underline{\textbf{Case $i=1$:}} We compute that
$$v_{1} = \frac{1}{\sqrt{a^2-b^2}}(\pm a,\pm ib).$$
A line passing through $S$ and directed by one of the solution for $v_2$ intersects the ellipse in $S$ and a point at infinity $M_{\pm}=(a:\pm ib:0)$. We get the orbit $(S,M_+,S',M_-)$ where $S'=-S$, see Figure \ref{fig_infinite_4_orbit}.

\begin{figure}[!h]
\centering
\definecolor{cqcqcq}{rgb}{0.7529411764705882,0.7529411764705882,0.7529411764705882}
\begin{tikzpicture}[line cap=round,line join=round,>=triangle 45,x=1cm,y=1cm]
\clip(-3,-2) rectangle (3,2);
\draw [samples=50,domain=-0.99:0.99,rotate around={0:(0,0)},xshift=0cm,yshift=0cm,line width=1pt] plot ({1.3694832979641993*(1+(\x)^2)/(1-(\x)^2)},{1.4575717809415425*2*(\x)/(1-(\x)^2)});
\draw [samples=50,domain=-0.99:0.99,rotate around={0:(0,0)},xshift=0cm,yshift=0cm,line width=1pt] plot ({1.3694832979641993*(-1-(\x)^2)/(1-(\x)^2)},{1.4575717809415425*(-2)*(\x)/(1-(\x)^2)});
\draw [line width=1pt,domain=-10.057909801527712:-1.3695961860964068] plot(\x,{(--8.149010140242552--6.024969698131765*\x)/-5.490945816698725});
\draw [line width=1pt,domain=-10.057909801527712:-1.3695961860964068] plot(\x,{(-8.345933921333584-6.0174292086845655*\x)/-5.582872525911064});
\draw [line width=1pt,domain=1.3696656621142531:11.063769561650973] plot(\x,{(-8.41061568482461--6.044491052779263*\x)/5.535817039638287});
\draw [line width=1pt,domain=1.3696656621142531:11.063769561650973] plot(\x,{(--8.052514052191567-5.974926176733984*\x)/5.512835362335203});
\begin{scriptsize}
\draw[color=black] (1.6,1.3) node {$\ellipse$};
\draw[color=black] (-1.05,0) node {$S$};
\draw[color=black] (1.05,0) node {$S'$};
\end{scriptsize}
\end{tikzpicture}
\caption{The infinite $4$-periodic orbit $T_1^4$. A change of coordinates $y\to iy$ to represent the the orbit properly ; this operation changes $\ellipse$ into a hyperbola. The infinite points of the orbit, $M_+$ and $M_-$, lie "at the end" of each branch of the hyperbola.}
\label{fig_infinite_4_orbit}
\end{figure}
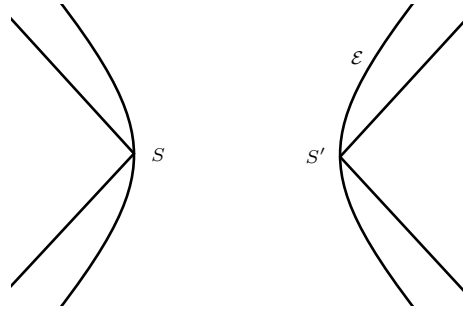

\noindent\underline{\textbf{Case $i=2$:}} By Equation \eqref{equ:v_from_lambda} and $q(v_2)=1$ we get 
$$v_{2} = \frac{1}{\sqrt{a^2+b^2}}(\pm a,\pm b).$$
It corresponds to the well-known real orbit $(S,P,S',P')$ where $P=(0,b)$, $P'=-P$ and $S'=-S$, see Figure \ref{fig_classical_4_orbit}. 

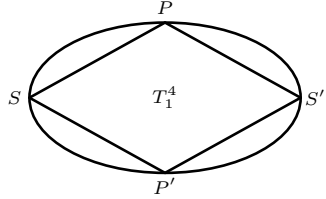
\begin{figure}[!h]
\centering
\begin{tikzpicture}[line cap=round,line join=round,>=triangle 45,x=0.5cm,y=0.5cm]
\clip(-4.5,-2.6) rectangle (4.5,2.6);
\draw [rotate around={0:(0,0)},line width=1pt] (0,0) ellipse (1.8028cm and 1cm);
\draw [line width=1pt] (-3.6055512754639882,0)-- (0,2);
\draw [line width=1pt] (3.6055512754639882,0)-- (0,2);
\draw [line width=1pt] (0,-2)-- (3.6055512754639882,0);
\draw [line width=1pt] (-3.6055512754639882,0)-- (0,-2);
\begin{scriptsize}
\draw[color=black] (-0,2.4) node {$P$};
\draw[color=black] (-4,0) node {$S$};
\draw[color=black] (4,0) node {$S'$};
\draw[color=black] (0,-2.35) node {$P'$};
\draw[color=black] (0,0) node {$T_1^4$};
\end{scriptsize}
\end{tikzpicture}
\caption{The $4$-periodic orbit $T_2^4$.}
\label{fig_classical_4_orbit}
\end{figure}

\noindent\underline{\textbf{Case $i=3$:}} Here we have
$$v_{3} = \frac{1}{\sqrt{a^2-b^2}}(\pm a,\pm i\sqrt{2a^2-b^2}).$$
A line passing through $S$ and directed by one of the solutions for $v_3$ intersects the ellipse in $S$ and in one of the points 
$$N_{\pm}=\frac{1}{a^2-b^2}(-a^3,\pm ib\sqrt{2a^2-b^2})$$
depending on the signs we choose for $v_3$'s coordinates.

The question is: why are the points $N_{\pm}$ not on the $x$- or $y$-axis by symmetry of the $4$-periodic orbit ? The reason is that the line passing through $S$ and directed by $v_3$ is reflected in the same line at one of the points $N_{\pm}$. Indeed, the system of equations in $v$
$$P(N_+,v)=P(T_3^4) \qquad\text{ and }\qquad q(v)=1$$
(see Lemma \ref{lemma_carac_billiard_2}) has a unique solution (up to multiplication by $-1$), which is in fact $v=v_3$. Thus the corresponding orbit is $(S,N_+,S,N_-)$, see Figure \ref{fig_complex_4_orbit}.

\begin{remark}
Observe that the points $N_{\pm}$ belong to $\ellipse$ and to $\mathcal{C}_3$. Indeed, $SN_+$ and $N_+S$ realize the two tangent lines to $\mathcal{C}_3$ (by Theorem \ref{main_theorem_caustics}), thus there is only one such tangent line and therefore $N_+\in\mathcal{C}_3$. The same argument is true with $N_-$.
\end{remark}

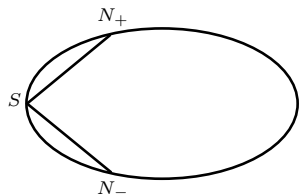
\begin{figure}[!h]
\centering
\definecolor{cqcqcq}{rgb}{0.7529411764705882,0.7529411764705882,0.7529411764705882}
\begin{tikzpicture}[line cap=round,line join=round,>=triangle 45,x=0.5cm,y=0.5cm]
\clip(-4.5,-2.6) rectangle (4.5,2.6);
\draw [rotate around={0:(0,0)},line width=1pt] (0,0) ellipse (1.8028cm and 1cm);
\draw [line width=1pt] (-3.588510623485942,0)-- (-1.3873680477621981,1.8159986473652563);
\draw [line width=1pt] (-3.588510623485942,0)-- (-1.3615197675189206,-1.8218796799437253);
\begin{scriptsize}
\draw[color=black] (-3.925012735702111,0.11464880614245819) node {$S$};
\draw[color=black] (-1.3,2.3) node {$N_+$};
\draw[color=black] (-1.3,-2.3) node {$N_-$};
\end{scriptsize}
\end{tikzpicture}
\caption{The $4$-periodic orbit $T_3^4$. Its edges $SN_+$ and $SN_-$ are reflected into themselves in $N_+$ and $N_-$ respectively. Here the drawing is biased, because we cannot represent $T_3^4$ which has complex vertices.}
\label{fig_complex_4_orbit}
\end{figure}

\newpage

\end{document}